\documentclass{amsart}

\usepackage{amsthm}
\usepackage{graphicx}
\usepackage{stmaryrd} 
\usepackage{tikz}
\usepackage{tikz-cd}
\usetikzlibrary{arrows}
\usepackage{ bbold }
\usepackage{verbatim}
\usepackage{amssymb,amsfonts}
\usepackage[all,arc]{xy}
\usepackage{enumerate}
\usepackage{mathrsfs, mathabx}
\usepackage{xcolor}[2007/01/21]
\usepackage{hyperref}
\usepackage{soul}
\usepackage{esint} 
\usepackage[all]{xy}
\usepackage[margin=1.0in]{geometry}
\usepackage[
   open,
   openlevel=2,
   atend
 ]{bookmark}[2011/12/02]
\usepackage{graphics}

\bookmarksetup{color=blue}

\theoremstyle{definition} 
\newtheorem{thm}{Theorem}[section]
\newtheorem{cor}[thm]{Corollary}

\newtheorem{lem}[thm]{Lemma}

\theoremstyle{definition}
\newtheorem{defn}[thm]{Definition}

\newtheorem{exmp}[thm]{Example}

\newtheorem{notn}[thm]{Notation}

\newtheorem{rmk}[thm]{Remark}

\theoremstyle{definition}

\theoremstyle{definition}

\theoremstyle{remark}

\newcommand{\GL}{\mathrm{GL}}

\newcommand{\Aut}{\mathrm{Aut}}

\newcommand{\bZ}{\mathbb{Z}}
\newcommand{\bQ}{\mathbb{Q}}

\newcommand{\bC}{\mathbb{C}}
\newcommand{\bF}{\mathbb{F}}
\newcommand{\bE}{\mathbb{E}}

\newcommand{\bb}{\mathbb}

\newcommand{\mr}{\mathrm}
\newcommand{\mf}{\mathfrak}
\newcommand{\ms}{\mathscr}
\newcommand{\mc}{\mathcal}
\newcommand{\sm}{\smallsetminus}
\newcommand{\sub}{\subset}

\newcommand{\sups}{\supset}

\newcommand{\Ann}{\mathrm{Ann}}

\newcommand{\Prob}{\mathrm{Prob}}

\newcommand{\Sur}{\mathrm{Sur}}

\newcommand{\M}{\mr{M}}

\newcommand{\ep}{\epsilon}

\newcommand{\ld}{\lambda}
\newcommand{\al}{\alpha}

\newcommand{\ga}{\gamma}
\newcommand{\ze}{\zeta}
\newcommand{\sg}{\sigma}
\newcommand{\dt}{\delta}

\newcommand{\Hom}{\mathrm{Hom}}
\newcommand{\Ext}{\mathrm{Ext}}

\newcommand{\id}{\mathrm{id}}
\newcommand{\bs}{\boldsymbol}

\newcommand{\ra}{\rightarrow}

\newcommand{\tra}{\twoheadrightarrow}

\newcommand{\hra}{\hookrightarrow}

\newcommand{\lt}{\left}
\newcommand{\rt}{\right}

\newcommand{\ot}{\otimes}
\newcommand{\op}{\oplus}

\newcommand{\Cl}{\text{Cl}}

\newcommand{\cok}{\mathrm{cok}}

\newcommand{\gm}{\hspace{-2mm}\mod}

\newcommand{\be}{\begin{enumerate}}
\newcommand{\ee}{\end{enumerate}}

\newcommand{\bi}{\begin{itemize}}
\newcommand{\ei}{\end{itemize}}

\newcommand{\bbm}{\begin{bmatrix}}
\newcommand{\ebm}{\end{bmatrix}}

\numberwithin{equation}{section}

\begin{document}

\title[the cokernel of a polynomial evaluated at a random integral matrix]{The distribution of the cokernel of a polynomial \\ evaluated at a random integral matrix}
\date{\today}
\author{Gilyoung Cheong and Myungjun Yu}
\address{G. Cheong -- Department of Mathematics, University of California, Irvine, 340 Rowland Hall, Irvine, California 92697, the United States of America \newline
M. Yu -- Department of Mathematics, Yonsei University, Seoul 03722, South Korea}
\email{gilyounc@uci.edu, mjyu@yonsei.ac.kr}

\begin{abstract}
Given a prime $p$, let $P(t)$ be a non-constant monic polynomial in $t$ over the ring $\bZ_{p}$ of $p$-adic integers. Let $X_{n}$ be an $n \times n$ random matrix over $\bZ_{p}$ with independent entries, each of which is not too concentrated on a single residue class modulo $p$. We prove that as $n \ra \infty$, the distribution of the cokernel $\cok(P(X_{n}))$ of $P(X_{n})$ converges to the distribution given by a finite product of some explicit measures that resemble Cohen--Lenstra measures. For example, the random matrix $X_{n}$ can be taken as a Haar-random matrix or a uniformly random $(0,1)$-matrix. We consider the distribution of $\cok(P(X_{n}))$ as a distribution of modules over $\bZ_{p}[t]/(P(t))$, which gives us a clearer formulation in comparison to considering the distribution as that of abelian groups. For the proof, we first reduce our problem into a problem over $\bZ/p^{k}\bZ$, for large enough positive integer $k$, in place of $\bZ_{p}$. Then we use a result of Sawin and Wood to reduce our problem into another problem of computing the limit of the expected number of surjective $(\bZ/p^{k}\bZ)[t]/(P(t))$-linear maps from $\cok(P(X_{n}))$ modulo $p^{k}$ to a fixed finite size $(\bZ/p^{k}\bZ)[t]/(P(t))$-module $G$. To estimate the expected number and compute the desired limit, we carefully adopt subtle techniques developed by Wood, which were originally used to compute the asymptotic distribution of the $p$-part of the sandpile group of a random graph.
\end{abstract}

\maketitle

\section{Introduction}

\hspace{3mm} We fix a prime $p$ and denote by $\M_{n}(A)$ the set of $n \times n$ matrices over a commutative ring $A$ with unity for $n \in \bZ_{\geq 1}$. In this paper, we study the distribution of the cokernel of a random matrix in $\M_{n}(\bZ_{p})$ as $n \ra \infty$, where $\bZ_{p}$ is the ring of $p$-adic integers. The earliest example was considered by Friedman and Washington \cite{FW}, who showed that for any finite abelian $p$-group $G$, we have
\begin{equation}\label{FW}
\lim_{n \ra \infty}\underset{{X \in \M_{n}(\bZ_{p})}}{\Prob}(\cok(X) \simeq G) = \frac{1}{|\Aut(G)|}\prod_{i=1}^{\infty}(1 - p^{-i}),
\end{equation}
where the probability is taken with respect to the Haar measure on $\M_{n}(\bZ_{p})$ and $\Aut(G)$ is the automorphism group of $G$. The right-hand side of the above identity defines a discrete probability measure on the set of of isomorphism classes of finite abelian $p$-groups, called the \textbf{Cohen--Lenstra measure}, coined by Cohen and Lenstra \cite{CL} to predict the distribution of the $p$-part of the class group $\Cl_{K}$ of a random imaginary quadratic extension $K$ of $\bQ$ for odd $p$. As noted by Venkatesh and Ellenberg \cite[Section 4.1]{VE}, the class group $\Cl_{K}$ of $K$ can be presented as the cokernel of a matrix in $\M_{n}(\bZ)$, where $n$ is any number of primes in the ring of integers of $K$ that generate $\Cl_{K}$. In particular, the $p$-part $\Cl_{K}[p^{\infty}]$ of the class group is the cokernel of a matrix in $\M_{n}(\bZ_{p})$. Hence, (\ref{FW}) provides a heuristic that the distribution of $\Cl_{K}[p^{\infty}]$ may be given by the Cohen--Lenstra measure when $K$ is chosen at random. For odd $p$, computing the distribution of $\Cl_{K}[p^{\infty}]$ is a long-standing conjecture in number theory.

\hspace{3mm} Motivated by this heuristic, Wood \cite{Woo19} extended (\ref{FW}) to a far more general class of probability measures on $\M_{n}(\bZ_{p})$. For example, her result \cite[Theorem 1.2]{Woo19} shows that (\ref{FW}) also holds for a uniformly random $(0,1)$-matrix $X \in \M_{n}(\bZ_{p})$, whose $(i,j)$-entries $X_{ij}$ are independent and each entry is defined as
\[X_{ij} = \left\{
	\begin{array}{ll}
	1 \mbox{ with probability } 1/2 \text{ and}  \\
	0 \mbox{ with probability } 1/2,
	\end{array}\right.\]
which is drastically different from an entry of a Haar-random matrix in $\M_{n}(\bZ_{p})$. We now give a definition due to Wood that includes both the Haar measure and the measure for a uniformly random $(0,1)$-matrix:

\begin{defn} Let $0 < \ep < 1$ be a real number. An \textbf{$\ep$-balanced} measure on $\bZ_{p}$ is a probability measure on the Borel $\sigma$-algebra or the discrete $\sg$-algebra of $\bZ_{p}$ with which
\[\underset{x \in \bZ_{p}}{\Prob}(x \equiv a \gm p) \leq 1 - \ep\]
for any $a \in \bF_{p}$. A probability measure on $\M_{n}(\bZ_{p}) = \bZ_{p}^{n^{2}}$ is said to be \textbf{$\ep$-balanced} if its random element has independent entries, each of which follows an $\ep$-balanced measure on $\bZ_{p}$.
\end{defn} 

\hspace{3mm} The Haar measure on $\M_{n}(\bZ_{p})$ with the Boral $\sg$-algebra is $\ep$-balanced with $\ep = 1 - 1/p$. The measure for a uniformly random $(0,1)$-matrix in $\M_{n}(\bZ_{p})$ with the discrete $\sg$-algebra is $\ep$-balanced with $\ep = 1/2$. From now on, we fix an arbitrary real number $0 < \ep < 1$. The $\sg$-algebra on $\M_n(\bZ_p)$ is assumed to be the Borel $\sg$-algebra or the discrete $\sg$-algebra.

\hspace{3mm} Wood \cite[Theorem 1.2]{Woo19} showed that (\ref{FW}) holds for any $\ep$-balanced measures on $(\M_{n}(\bZ_{p}))_{n \in \bZ_{\geq 1}}$. We generalize Wood's result to study the distribution of the cokernel $\cok(P(X))$ of the polynomial push-forward $P(X)$ of a random matrix $X \in \M_{n}(\bZ_{p})$ with an $\ep$-balanced measure, where $P(t) \in \bZ_{p}[t]$ is a monic polynomial. It is extremely important to note that $\cok(P(X))$ is not just an abelian group but an abelian group with an additional structure unless $\deg(P) = 1$. That is, we note that $\cok(P(X))$ is a module over $\bZ_{p}[t]/(P(t))$, where the action of the image $\bar{t}$ of $t \in \bZ_{p}[t]$ is given by the left-multiplication of $X$.

\begin{notn} Let $R$ be a commutative ring with unity. Given $R$-modules $U$ and $V$, we write $U \simeq_{R} V$ to mean that $U$ and $V$ are isomorphic as $R$-modules. In particular, we write $U \simeq_{\bZ} V$ to mean that $U$ and $V$ are isomorphic as abelian groups. We note that $U \simeq_{\bZ} V$ does not always imply $U \simeq_{R} V$. When $U$ and $V$ are $R/I$-modules for some ideal $I \sub R$, having $U \simeq_{R/I} V$ is equivalent to $U \simeq_{R} V$, so we may use either notation in such a case.

\hspace{3mm} We denote by $\Aut_{R}(G)$ the group of $R$-linear automorphisms of an $R$-module $G$. We write $\Hom_{R}(U, V)$ to mean the set of $R$-linear homomorphisms from $U$ to $V$. We write $\Ext_{R}^{i}(U, V)$ to mean the $i$-th Ext module over $R$.
\end{notn}

\hspace{3mm} We state our main theorem:

\begin{thm}\label{main} Let $P(t) \in \bZ_{p}[t]$ be a non-constant monic polynomial. Consider the unique factorization of the reduction $\bar{P}(t)$ of $P(t)$ modulo $p$ as follows: 
\[\bar{P}(t) = \bar{P}_{1}(t)^{m_{1}} \cdots \bar{P}_{l}(t)^{m_{l}},\]
where $\bar{P}_{j}(t) \in \bF_{p}[t]$ are distinct monic irreducible polynomials and $m_{j} \in \bZ_{\geq 1}$. We write $d_{j} := \deg(\bar{P}_{j})$. For any $\ep$-balanced measures on $(\M_{n}(\bZ_{p}))_{n \in \bZ_{\geq 1}}$ and any finite size module $G$ over $\bZ_{p}[t]/(P(t))$, we have
\[\lim_{n \ra \infty}\underset{X \in \M_{n}(\bZ_{p})}{\Prob}(\cok(P(X)) \simeq_{\bZ_{p}[t]} G)
= \dfrac{1}{|\Aut_{\bZ_{p}[t]}(G)|}\displaystyle\prod_{j=1}^{l}\displaystyle\prod_{i=1}^{\infty}\lt(1 - \frac{|\Ext_{\bZ_{p}[t]/(P(t))}^{1}(G, \bF_{p^{d_{j}}})| p^{-id_{j}}}{|\Hom_{\bZ_{p}[t]}(G, \bF_{p^{d_{j}}})|}\rt),\]
where $\bF_{p^{d_{j}}} := \bF_{p}[t]/(\bar{P}_{j}(t))$, a finite field of $p^{d_j}$ elements.
\end{thm}

\begin{rmk} It is interesting to note that $m_{1}, \dots, m_{l}$ do not appear on the right-hand side of the conclusion of Theorem \ref{main}. The information about $m_{1}, \dots, m_{l}$ is incorporated in $|\Ext_{\bZ_{p}[t]/(P(t))}^{1}(G, \bF_{p^{d_{j}}})|$. For example, when $m_{j} = 1$, we have $|\Ext_{\bZ_{p}[t]/(P(t))}^{1}(G, \bF_{p^{d_{j}}})| = |\Hom_{\bZ_{p}[t]}(G, \bF_{p^{d_{j}}})|$, as we show in Section \ref{ext}. It turns out that we always have
\[|\Hom_{\bZ_{p}[t]}(G, \bF_{p^{d_{j}}})| \leq |\Ext_{\bZ_{p}[t]/(P(t))}^{1}(G, \bF_{p^{d_{j}}})|,\]
and we learned from Will Sawin that when the above inequality is strict, the probability in Theorem \ref{main} becomes $0$ (which we explain in Lemma \ref{vanishing}).

\hspace{3mm} We also note that this equality may not be achieved when $m_{j} > 1$. For example, we see in Example \ref{extexmp} that when $P(t) = t^{2}$ and $G = \bF_{p}[t]/(t) = \bF_{p}$, we get $|\Ext_{\bZ_{p}[t]/(P(t))}^{1}(G, \bF_{p})| = p|\Hom_{\bZ_{p}[t]}(G, \bF_{p})| > |\Hom_{\bZ_{p}[t]}(G, \bF_{p})|$. In particular, we have 
\[\lim_{n \ra \infty}\underset{X \in \M_{n}(\bZ_{p})}{\Prob}(\cok(X^{2}) \simeq_{\bZ_{p}[t]} \bF_{p}) = 0.\]
\end{rmk}

\hspace{3mm} Following the above remark, if the reduction of $P(t)$ modulo $p$ is square-free in $\bF_{p}[t]$, then Theorem \ref{main} gives the following:

\begin{thm}\label{sqfree} Let $P(t) \in \bZ_{p}[t]$ be a non-constant monic polynomial whose reduction modulo $p$ is square-free in $\bF_{p}[t]$. Consider the unique factorization of the reduction $\bar{P}(t)$ of $P(t)$ modulo $p$ as follows: 
\[\bar{P}(t) = \bar{P}_{1}(t) \cdots \bar{P}_{l}(t),\]
where $\bar{P}_{j}(t) \in \bF_{p}[t]$ are distinct monic irreducible polynomials. We write $d_{j} := \deg(\bar{P}_{j})$. For any $\ep$-balanced measures on $(\M_{n}(\bZ_{p}))_{n \in \bZ_{\geq 1}}$ and any finite size module $G$ over $\bZ_{p}[t]/(P(t))$, we have
\[\lim_{n \ra \infty}\underset{{X \in \M_{n}(\bZ_{p})}}{\Prob}(\cok(P(X)) \simeq_{\bZ_{p}[t]} G) = \frac{1}{|\Aut_{\bZ_{p}[t]}(G)|} \prod_{j=1}^{l}\prod_{i=1}^{\infty}(1 - p^{-id_{j}}).\]
\end{thm}

\begin{rmk} In Theorems \ref{main} and \ref{sqfree}, it is important that we consider the condition $\cok(P(X)) \simeq_{\bZ_p[t]} G$ as an isomorphism of modules over $\bZ_{p}[t]$ (or equivalently, over $\bZ_{p}[t]/(P(t))$), not just $\cok(P(X)) \simeq G$, an isomorphism of abelian groups. The following corollary shows how the statement of Theorem \ref{sqfree}, let alone Theorem \ref{main}, becomes more convoluted if we consider $\cok(P(X)) \simeq G$ as abelian groups.
\end{rmk}

\hspace{3mm} Given a partition $\ld = (\ld_{1}, \cdots, \ld_{r})$, we write $H_{\ld} := \bZ/p^{\ld_{1}}\bZ \times \cdots \times \bZ/p^{\ld_{r}}\bZ$. (We always adopt the convention that $\ld_{1} \geq \cdots \geq \ld_{r}$.)

\begin{cor}\label{cor1} Let $P(t) \in \bZ_{p}[t]$ be a non-constant monic polynomial whose reduction modulo $p$ is square-free in $\bF_{p}[t]$. Consider the unique factorization of the reduction $\bar{P}(t)$ of $P(t)$ modulo $p$ as follows: 
\[\bar{P}(t) = \bar{P}_{1}(t) \cdots \bar{P}_{l}(t),\]
where $\bar{P}_{j}(t) \in \bF_{p}[t]$ are distinct monic irreducible polynomials. We write $d_{j} := \deg(\bar{P}_{j})$. For any $\ep$-balanced measures on $(\M_{n}(\bZ_{p}))_{n \in \bZ_{\geq 1}}$ and any finite size module $G$ over $\bZ_{p}[t]/(P(t))$, we have
\[\lim_{n \ra \infty}\underset{{X \in \M_{n}(\bZ_{p})}}{\Prob}(\cok(P(X)) \simeq_{\bZ} G) = \frac{1}{|\Aut_{\bZ_{p}[t]}(G)|} \sum_{\substack{(\ld^{(1)}, \dots, \ld^{(l)}): \\ H_{\ld^{(1)}}^{d_{1}} \times \cdots \times H_{\ld^{(l)}}^{d_{l}} \simeq_{\bZ} G}} \prod_{j=1}^{l}\prod_{i=1}^{\infty}(1 - p^{-id_{j}}),\]
where each $(\ld^{(1)}, \dots, \ld^{(l)})$ we sum over is an ordered tuple of partitions $\ld^{(j)}$ such that $H_{\ld^{(1)}}^{d_{1}} \times \cdots \times H_{\ld^{(l)}}^{d_{l}} \simeq_{\bZ} G$.
\end{cor}

\hspace{3mm} It is surprising how succinct the formulations of Theorems \ref{main} and \ref{sqfree} are in comparison to Corollary \ref{cor1}. This reflects the Wood's philosophy \cite[p.23]{Woo22} that when the cokernel of a random matrix is given an additional structure on top of the abelian group structure, the distribution of it must take into account this structure. To see how Theorem \ref{sqfree} implies Corollary \ref{cor1}, we first note that Theorem \ref{sqfree} can be immediately reformulated as follows:

\begin{cor}\label{cor2}  Let $P(t) \in \bZ_{p}[t]$ be a non-constant monic polynomial whose reduction modulo $p$ is square-free in $\bF_{p}[t]$ so that by Hensel's lemma, we have
\[P(t) = P_{1}(t) \cdots P_{l}(t),\]
for some monic polynomials $P_{1}(t), \dots, P_{l}(t) \in \bZ_{p}[t]$ whose reductions modulo $p$ in $\bF_{p}[t]$ are distinct and irreducible. For $1 \leq j \leq l$, fix any finite size module $G_{j}$ over $\bZ_{p}[t]/(P_{j}(t))$. For any $\ep$-balanced measures on $(\M_{n}(\bZ_{p}))_{n \in \bZ_{\geq 1}}$, we have
\[\lim_{n \ra \infty}\underset{X \in \M_{n}(\bZ_p)}{\Prob}\left(\begin{array}{c}
\cok(P_{j}(X)) \simeq G_{j} \\
\text{for } 1 \leq j \leq l
\end{array}\right)
= \displaystyle\prod_{j=1}^{l}\dfrac{1}{|\Aut_{\bZ_{p}[t]}(G_{j})|}\displaystyle\prod_{i=1}^{\infty}(1 - p^{-i\deg(P_{j})}),\]
where $\simeq$ can be either $\simeq_{\bZ}$ or $\simeq_{\bZ_{p}[t]}$.
\end{cor}

\hspace{3mm} The reason that we can consider $\cok(P_{j}(X)) \simeq G_{j}$ appearing in Corollary \ref{cor2} as either an isomorphism of abelian groups or modules over $\bZ_{p}[t]$ (or equivalently, over $\bZ_{p}[t]/(P_{j}(t))$) is because $\bZ_{p}[t]/(P_{j}(t))$ is a discrete valuation ring (DVR) whose maximal ideal is generated by $p$. That is, any finite size module over it is a finite product of modules of the form 
\[(\bZ/p^{k}\bZ)[t]/(P_{j}(t)) = (\bZ/p^{k}\bZ) \op \bar{t} (\bZ/p^{k}\bZ) \op \cdots \op \bar{t}^{\deg(P_{j})-1}(\bZ/p^{k}\bZ) \simeq_{\bZ} (\bZ/p^{k}\bZ)^{d_{j}}\]
with various $k \in \bZ_{\geq 1}$, so any two finite size $\bZ_{p}[t]/(P_{j}(t))$-modules are isomorphic as $\bZ_{p}[t]/(P_{j}(t))$-modules if and only if they are isomorphic as abelian groups. Taking $G_{j} = H_{\ld^{(j)}}^{d_{j}} = H_{\ld^{(j)}} \op \bar{t}H_{\ld^{(j)}} \op \cdots \op \bar{t}^{d_{j}-1}H_{\ld^{(j)}}$, Corollary \ref{cor2} implies Corollary \ref{cor1}, and thus Theorem \ref{sqfree} implies Corollary \ref{cor1}.

\hspace{3mm} An interesting special case of Corollary \ref{cor2} is when we take $P_{1}(t) = t$ and $P_{2}(t) = t - 1$ with $G_{1} = 0$ and $G_{2} = G$ for any finite abelian $p$-group $G$, which gives us 
\begin{equation}\label{GL}
\lim_{n \ra \infty}\underset{{X \in \GL_{n}(\bZ_{p})}}{\Prob}(\cok(X - I_{n}) \simeq G) = \frac{1}{|\Aut(G)|}\prod_{i=1}^{\infty}(1 - p^{-i}),
\end{equation}
where the probability measure on $\GL_{n}(\bZ_{p})$ is induced by the given $\ep$-balanced measure on $\M_{n}(\bZ_{p})$ and $I_{n}$ is the $n \times n$ identity matrix. For the Haar measures, the conclusion for (\ref{GL}) is due to Friedman and Washington \cite{FW}.

\subsection{Related works, crucial methods, and relevant viewpoints} Theorem \ref{main} (or Theorem \ref{sqfree}) is a generalization of Wood's theorem \cite[Theorem 1.2]{Woo19} by taking $P(t) = t$. Wood's result was first generalized by Lee \cite[Theorem 1.5]{LeeB} which corresponds to the case $d_{1} = \cdots = d_{l} = 1$ and $m_{1} = \cdots = m_{l} = 1$ for Theorem \ref{main}. 

\hspace{3mm} When $(\M_{n}(\bZ_{p}))_{n \in \bZ_{\geq 1}}$ are equipped with the Haar measures, the conclusion of Corollary \ref{cor2} was first conjectured by the first author and Huang in \cite{CH}, where the authors proved the case where $G_{1} = \cdots = G_{l-1} = 0$ and $d_{l} = 1$. The first author and Kaplan \cite{CK} proved the case where $d_{1}, \dots, d_{l} \leq 2$, and for general $d_{1}, \dots, d_{l}$ Corollary \ref{cor2} for the Haar measures was proven by Lee \cite[Theorem 1.2]{LeeA}. When $\M_{n}(\bZ_{p})$ is given the Haar measure, there is even an explicit formula for
\[\underset{X \in \M_{n}(\bZ_p)}{\Prob}\left(\begin{array}{c}
\cok(P_{j}(X)) \simeq G_{j} \\
\text{for } 1 \leq j \leq l \\
\text{and } X \equiv \bar{X} \pmod{p}
\end{array}\right)\]
for every fixed $n \in \bZ_{\geq 1}$ and $\bar{X} \in \M_{n}(\bF_{p})$ for many special cases of $P(t)$. This explicit formula is currently known for $d_{1}, \dots, d_{l} \leq 2$ due to the first author and Kaplan \cite{CK} and for $l = 1$ due to the first author, Liang and Strand \cite[Corollary 1.10]{CLS}. The explicit formula for the case $l=1$ and $d_{l} = 1$ was first known due to Friedman and Washington \cite{FW}. 

\hspace{3mm} When the reduction of $P(t)$ modulo $p$ is not square-free in $\bF_{p}[t]$, Theorem \ref{main} is new even with the Haar measures. Although the Haar measures are more accessible for explicit computations, for a general monic polynomial $P(t) \in \bZ_{p}[t]$, the formulation of Theorem \ref{main} is extremely complicated when we do not consider the distribution $(\cok(P(X))_{X \in \M_{n}(\bZ_{p})}$ as a distribution of finite size $\bZ_{p}[t]/(P(t))$-modules instead of that of finite abelian $p$-groups, but it is also difficult to check when $\cok(P(X))$ is isomorphic to a given module over $\bZ_{p}[t]/(P(t))$, not just as abelian groups.

\hspace{3mm} The distribution of the cokernel of an $\ep$-balanced random integral matrix was first considered by Wood in \cite{Woo17}, a breakthrough paper where she computed the asymptotic distribution of the $p$-part of the sandpile group of an Erd\H{o}s--R\'enyi random graph on $n$ vertices with constant independent edge probabilities as $n \ra \infty$ \cite[Theorem 1.1]{Woo17}. In the same paper, Wood also used similar methods to prove an analogous theorem to (\ref{FW}) for an $n \times n$ random $\ep$-balanced symmetric matrix \cite[Theorem 1.3]{Woo17}. Wood's methods from \cite{Woo17} turned out to be extremely pivotal in solving various related problems. In addition to the already mentioned works \cite{Woo19} and \cite{LeeB}, Wood's methods were applied to compute, as $n \ra \infty$, the asymptotic distribution of the cokernel of an $n \times n$ random $\ep$-balanced Hermitian matrix over a quadratic extension of $\bZ_{p}$ \cite{LeeC}, the reduced Laplacian over $\bZ_{p}$ of a random regular graph \cite{Mes}, the product of a fixed number of $n \times n$ random $\ep$-balanced matrices over $\bZ_{p}$ \cite{NV}, the adjacency matrix over $\bZ_{p}$ of a random regular graph \cite{NW18}, and an $n \times (n + u)$ $\ep$-balanced random integral matrix with $u \in \bZ_{\geq 1}$ \cite{NW22}.

\hspace{3mm} As in many works mentioned above, our proof of Theorem \ref{main} heavily relies on Wood's methods from \cite{Woo17, Woo19}, with which we compute the limit of the expected number of surjective $\bZ_{p}[t]/(P(t))$-linear maps from $\cok(P(X))$ to an arbitrary finite size module $G$ for a random matrix $X \in \M_{n}(\bZ_{p})$ as $n \ra \infty$. (The expected number is called the ``$G$-moment'' and it is introduced in \S \ref{moments}.) However, many theorems in \cite{Woo17} are developed for abelian groups, not $\bZ_{p}[t]/(P(t))$-modules, so there are subtle places where we need check to make sure that the techniques can be applied to our setting. After computing the limit of the $G$-moment, which turns out to be $1$, it still requires a significant amount of work to show that such a limit implies Theorem \ref{main}. Thankfully, this job is already done by a result of Sawin and Wood \cite[Lemma 6.3]{SW}, a special case of which we restate as Lemma \ref{SW}.

\hspace{3mm} In our proof, we also use an important insight due to Lee \cite{LeeA}, which tells us that when we study the distribution of $\cok(P(X))$ for a random matrix $X \in \M_{n}(\bZ_{p})$, we may use the $\bZ_{p}[t]/(P(t))$-linear isomorphism
\[\cok(P(X)) = \frac{\bZ_{p}^{n}}{P(X)\bZ_{p}^{n}} \simeq \frac{(\bZ_{p}[t]/(P(t)))^{n}}{(X - \bar{t}I_{n})(\bZ_{p}[t]/(P(t)))^{n}} =: \cok_{\bZ_{p}[t]/(P(t))}(X - \bar{t}I_{n}),\]
induced by the inclusion $\bZ_{p}^{n} \hra (\bZ_{p}[t]/(P(t)))^{n}$, whose image consists of tuples of constant polynomials modulo $P(t)$. This isomorphism linearizes our problem by letting us resolve the difficulty of taking the polynomial push-forward $P(X)$ of $X$ by dealing with a more complicated ring $\bZ_{p}[t]/(P(t))$ instead of $\bZ_{p}$. This also gives another perspective on studying the distribution of the cokernel of a random matrix in $\M_{n}(\bZ_{p}[t]/(P(t)))$. When the reduction of $P(t)$ modulo $p$ is irreducible in $\bF_{p}[t]$, we can compare Theorem \ref{main} to a result of Yan \cite[Theorem 1.2]{Yan}:
\begin{thm}[Yan]\label{Yan} Let $P(t) \in \bZ_{p}[t]$ be a non-constant monic polynomial whose reduction modulo $p$ is irreducible in $\bF_{p}[t]$. For each $n \in \bZ_{\geq 1}$, suppose that $\M_{n}(\bZ_{p}[t]/(P(t)))$ is given the probability measure, where a random matrix in $\M_{n}(\bZ_{p}[t]/(P(t)))$ has independent entries, each of which follows a probability measure on the Borel $\sigma$-algebra of $\bZ_{p}[t]/(P(t))$ such that
\[\underset{y \in \bZ_{p}[t]/(P(t))}{\Prob}(y \in H \gm p) \leq 1 - \ep\]
for every proper affine $\bF_{p}$-subspace $H$ of $\bF_{p}[t]/(P(t))$. Then for any finite size module $G$ over $\bZ_{p}[t]/(P(t))$, we have
\[\lim_{n \ra \infty}\underset{Y \in \M_{n}(\bZ_{p}[t]/(P(t)))}{\Prob}(\cok_{\bZ_{p}[t]/(P(t))}(Y) \simeq G) = \frac{1}{|\Aut_{\bZ_{p}[t]}(G)|}\prod_{i=1}^{\infty}(1 - p^{-i\deg(P)}),\]
where $\simeq$ can be either taken as $\simeq_{\bZ}$ or $\simeq_{\bZ_{p}[t]}$.
\end{thm} 

\hspace{3mm} Given the same hypothesis as in Theorem \ref{Yan}, Theorem \ref{main} (or Theorem \ref{sqfree}) states that
\begin{equation}\label{CY}
\lim_{n \ra \infty}\underset{{X \in \M_{n}(\bZ_{p})}}{\Prob}(\cok_{\bZ_{p}[t]/(P(t))}(X - \bar{t}I_{n}) \simeq G) = \frac{1}{|\Aut_{\bZ_{p}[t]}(G)|}\prod_{i=1}^{\infty}(1 - p^{-i\deg(P)}),
\end{equation}
where $\simeq$ can be either taken as $\simeq_{\bZ}$ or $\simeq_{\bZ_{p}[t]}$ because $\bZ_{p}[t]/(P(t))$ is a DVR with the maximal ideal $(p)$. When $\deg(P) \geq 2$, we may view this as taking a random matrix 
\[Y = X_{0} + \bar{t}X_{1} + \cdots + \bar{t}^{\deg(P)-1}X_{\deg(P)-1}\]
in
\[\M_{n}(\bZ_{p}[t]/(P(t))) = \M_{n}(\bZ_{p}) \op \bar{t}\M_{n}(\bZ_{p}) \op \cdots \op \bar{t}^{\deg(P)-1}\M_{n}(\bZ_{p})\]
with $X_{0} \in \M_{n}(\bZ_{p})$ and $X_{1} = - I_{n}$, while any $X_{j} = 0$ for $j \geq 2$. Hence, every diagonal entry of $Y$ modulo $p$ sits inside
\[H := \{a_{0} + a_{1}\bar{t} + a_{2}\bar{t}^{2} \cdots + a_{\deg(P)-1}\bar{t}^{\deg(P)-1} : a_{0} \in \bF_{p} \text{ while } a_{1} = -1 \text{ and } a_{i} = 0 \text{ for } 2 \leq i \leq \deg(P)-1\},\]
which is a proper affine $\bF_{p}$-subspace of $\bF_{p}[t]/(P(t)) = \bF_{p} \op \bar{t}\bF_{p} \op \cdots \op \bar{t}^{\deg(P)-1}\bF_{p}$. In particular, (\ref{CY}) has the same conclusion as in Theorem \ref{Yan} although the way we pick our random matrix $Y \in \M_{n}(\bZ_{p}[t]/(P(t)))$ is not covered by the hypothesis of the cited theorem. Of course, this is only a special case of Yan's work \cite{Yan}, which aims to capture the correct notion of $\ep$-balancedness over a more general DVR or a Dedekind domain. The main point of our comparison between Theorem \ref{Yan} and (\ref{CY}) is to suggest that there is still room for generalizations in this direction of replacing $\bZ_{p}$ with a more general DVR beyond the scope of \cite{Yan}.

\begin{rmk} However, such a generalization needs much care. For example, consider $P(t)$ and $G$ as in Theorem \ref{Yan}, but let us equip $(\M_{n}(\bZ_{p}))_{n \in \bZ_{\geq 1}}$ with the Haar measures. In this case, it is \emph{not true} that
\begin{equation}\label{F}
\lim_{n \ra \infty}\underset{{X \in \M_{n}(\bZ_{p})}}{\Prob}(\cok_{\bZ_{p}[t]/(P(t))}(X) \simeq G) = \frac{1}{|\Aut_{\bZ_{p}[t]}(G)|}\prod_{i=1}^{\infty}(1 - p^{-i\deg(P)})
\end{equation}
unless $\deg(P) = 1$. (Here, the notation $\simeq$ can be taken to be either $\simeq_{\bZ}$ or $\simeq_{\bZ_{p}[t]}$.) This is because since all of the entries of $X$ are in $\bZ_{p}$, we have
\[\cok_{\bZ_{p}[t]/(P(t))}(X) = \frac{(\bZ_{p}[t]/(P(t)))^{n}}{X(\bZ_{p}[t]/(P(t)))^{n}} = \frac{\bZ_{p}^{n} \op \bar{t}\bZ_{p}^{n} \op \cdots \op \bar{t}^{d-1}\bZ_{p}^{n}}{X(\bZ_{p}^{n} \op \bar{t}\bZ_{p}^{n} \op \cdots \op \bar{t}^{d-1}\bZ_{p}^{n})} \simeq_{\bZ} \lt(\frac{\bZ_{p}^{n}}{X\bZ_{p}^{n}}\rt)^{d} = \cok(X)^{d},\]
where $d := \deg(P)$. In particular, if $G = \bF_{p}[t]/(P(t))$, a $d$-dimensional $\bF_{p}$-vector space, then the identity (\ref{F}) yields a contradiction for $d > 1$ because $\cok_{\bZ_{p}[t]/(P(t))}(X) \simeq G$ if and only if $\cok(X) \simeq_{\bZ} \bF_{p}$ so that
\[\lim_{n \ra \infty}\underset{{X \in \M_{n}(\bZ_{p})}}{\Prob}(\cok_{\bZ_{p}[t]/(P(t))}(X) \simeq G) = \lim_{n \ra \infty}\underset{{X \in \M_{n}(\bZ_{p})}}{\Prob}(\cok(X) \simeq_{\bZ} \bF_{p}) = \frac{1}{p-1}\prod_{i=1}^{\infty}(1 - p^{-i\deg(P)})\]
while we have $|\Aut_{\bZ_{p}[t]}(G)| = p^{d} - 1$. This example was pointed out by Jungin Lee, in a previous communication with the first author.
\end{rmk}

\subsection{Working modulo a fixed prime power}\label{power} Recall that the ring $\bZ_{p}$ of the $p$-adic integers is the inverse limit of the system $\cdots \tra \bZ/p^{3}\bZ \tra \bZ/p^{2}\bZ \tra \bZ/p\bZ = \bF_{p}$ of projections, so it is often possible to reduce a problem over $\bZ_{p}$ into a problem over $\bZ/p^{k}\bZ$ for large enough $k \in \bZ_{\geq 1}$. We can also do this for the proof of Theorem \ref{main}. That is, we have $\cok(P(X)) \simeq_{\bZ_{p}[t]} G$ if and only if $\cok(P(X')) \simeq_{(\bZ/p^{k}\bZ)[t]} G$, where $X' \in \M_{n}(\bZ/p^{k}\bZ)$ is the image of $X$ modulo $p^{k}$. (See Lemma \ref{red}.)

\begin{rmk} For a general monic polynomial $P(t) \in \bZ_{p}[t]$, we do not know whether having $\cok(Y') \simeq_{(\bZ/p^{k}\bZ)[t]} G$ implies $\cok(Y) \simeq_{\bZ_{p}[t]} G$ for arbitrary $Y \in \M_{n}(\bZ_{p}[t]/(P(t)))$, where $Y' \in \M_{n}((\bZ/p^{k}\bZ)[t]/(P(t))$ is the image of $Y$ modulo $p^{k}$. For example, we do not have any classification result for finitely generated modules over $\bZ_{p}[t]/(P(t))$, and in particular, we do not have any analogue of the Smith normal form of $Y$, as for the case of matrices over a PID (or quotients of it). However, in our case, we have a very special $Y$, namely $Y = X - \bar{t}I_{n}$ with $X \in \M_{n}(\bZ_{p})$ so that
\[\cok_{\bZ_{p}[t]/(P(t))}(Y) = \cok_{\bZ_{p}[t]/(P(t))}(X - \bar{t}I_{n}) \simeq_{\bZ_{p}[t]} \cok(P(X)).\]
Hence, we can use the Smith normal form of $P(X)$ over $\bZ_{p}$ to resolve this issue. (More details are given in the proof of Lemma \ref{red}.)
\end{rmk}

\hspace{3mm} We are given an $\ep$-balanced measure on $\M_{n}(\bZ_{p})$, and the projection $\M_{n}(\bZ_{p}) \tra \M_{n}(\bZ/p^{k}\bZ)$ modulo $p^{k}$ induces a probability measure on $\M_{n}(\bZ/p^{k}\bZ) = (\bZ/p^{k}\bZ)^{n^{2}}$ given by an $n^{2}$-fold product of discrete probability measures on each $\bZ/p^{k}\bZ$ such that for every $a \in \bF_{p}$, we have
\[\underset{x \in \bZ/p^{k}\bZ}{\Prob}(x \equiv a \gm p) \leq 1 - \ep.\]
We shall also say that a discrete probability measure on $\M_{n}(\bZ/p^{k}\bZ)$ satisfying the above property is \textbf{$\ep$-balanced}. All in all, Theorem \ref{main} is equivalent to the statement obtained by replacing $\bZ_{p}$ with $\bZ/p^{k}\bZ$ for any $k \in \bZ_{\geq 1}$ such that $p^{k-1}G = 0$.

\subsection{Moments}\label{moments} Given $k \in \bZ_{\geq 1}$, let $R := (\bZ/p^{k}\bZ)[t]/(P(t))$, where $P(t) \in (\bZ/p^{k}\bZ)[t]$ is a monic polynomial. Given a finite size $R$-module $G$, the expected number 
\[\underset{X \in \M_{n}(\bZ/p^{k}\bZ)}{\bE}|\Sur_{R}(\cok(P(X)), G)|\]
of $R$-linear maps from $\cok(P(X))$ onto $G$ is called the \textbf{$G$-moment} for the distribution $(\cok(P(X))_{X \in \M_{n}(\bZ/p^{k}\bZ)}$ of finite size $R$-modules, where $X \in \M_{n}(\bZ/p^{k}\bZ)$ is chosen at random with a given $\ep$-balanced measure. Thanks to a recent work of Sawin and Wood \cite[Lemma 6.3]{SW}, which we restate as Theorem \ref{SW}, to prove Theorem \ref{main} (or technically, the equivalent version where we replace $\bZ_{p}$ with $\bZ/p^{k}\bZ$ for large enough $k \in \bZ_{\geq 1}$), it is enough to show that an arbitrary $G$-moment converges to $1$ as $n \ra \infty$. Hence, using that $\cok(P(X)) \simeq_{R} \cok_{R}(X - \bar{t}I_{n})$  (e.g., using \cite[Lemma 3.2]{CK}) the rest of the paper focuses on showing the following theorem:

\begin{thm}\label{main2} Let $G$ be any finite size $R$-module. Then
\[\lim_{n \ra \infty}\underset{X \in \M_{n}(\bZ/p^{k}\bZ)}{\bE}|\Sur_{R}(\cok_{R}(X - \bar{t}I_{n}), G)| = 1.\]
\end{thm}

\begin{rmk} Choose $k \in \bZ_{\geq 1}$ so that $p^{k-1}G = 0$, where $G$ is a finite size $\bZ_{p}[t]/(P(t))$-module. Writing $R := (\bZ/p^{k}\bZ)[t]/(P(t))$, we have
\[\underset{X \in \M_{n}(\bZ_{p})}{\bE}|\Sur_{\bZ_{p}[t]/(P(t))}(\cok(P(X)), G)| = \underset{X \in \M_{n}(\bZ/p^{k}\bZ)}{\bE}|\Sur_{R}(\cok_{R}(X - \bar{t}I_{n}), G)|,\]
if we are given an $\ep$-balanced measure on $\M_{n}(\bZ_{p})$, which induces an $\ep$-balanced measure on $\M_{n}(\bZ/p^{k}\bZ)$. Hence, the discussion of using the expected number of $\bZ_{p}[t]/(P(t))$-surjections in the earlier part of the introduction of this paper is consistent with the current discussion.
\end{rmk}


\section{Size of $\Ext^{1}_{\bZ_{p}[t]/(P(t))}(G, \bF_{p^{d_j}})$}\label{ext}

\hspace{3mm} In this section, we discuss the size of $\Ext^{1}_{\bZ_{p}[t]/(P(t))}(G, \bF_{p^{d_j}})$ appearing in Theorem \ref{main}. We are given a monic polynomial $P(t) \in \bZ_{p}[t]$ and a finite size module $G$ over $\bZ_{p}[t]/(P(t))$. Let $\bar{P}(t) \in \bF_{p}[t]$ be the reduction of the given monic polynomial $P(t) \in \bZ_{p}[t]$ modulo $p$ and consider the unique factorization 
\begin{equation}\label{fac}
\bar{P}(t) = \bar{P}_1(t)^{m_1} \bar{P}_2(t)^{m_2} \cdots  \bar{P}_l(t)^{m_l}
\end{equation}
in $\bF_p[t]$, where $\bar{P}_{1}(t), \dots, \bar{P}_{l}(t)$ are distinct monic irreducible polynomials in $\bF_{p}[t]$ and $m_{1}, \dots, m_{l} \in \bZ_{\geq 1}$. We write $d_{j} := \deg(\bar{P}_{j})$ as in Theorem \ref{main}. By Hensel's lemma, we have
\[P(t) = Q_1(t) Q_2(t) \cdots  Q_l(t),\]
where $Q_{j}(t) \in \bZ_{p}[t]$ is a monic polynomial whose reduction modulo $p$ is $\bar{Q}_{j}(t) = \bar{P}_{j}(t)^{m_j}$. Note that the principal ideals $(Q_{i}(t))$ and $(Q_{j}(t))$ are comaximal in $\bZ_{p}[t]$ whenever $i \neq j$ by Nakayama's lemma. We recall that $\bF_{p^{d_j}} := \bF_{p}[t]/(\bar{P}_{j}(t))$, a finite field of $p^{d_{j}}$ elements.

\

\hspace{3mm} Fix $k \in \bZ_{\geq 1}$ such that $p^{k-1}G = 0$. Then $G$ is a module over $R := (\bZ/p^k\bZ)[t]/(P(t))$. (There is an abuse of notation: $P(t)$ also means the image of $P(t) \in \bZ_{p}[t]$ in $(\bZ/p^k\bZ)[t]$.)

\begin{lem}\label{extred} Keeping the notation above, for any $1 \leq j \leq l$, we have
\[|\Ext^{1}_{\bZ_{p}[t]/(P(t))}(G, \bF_{p^{d_j}})| = |\Ext^{1}_{R}(G, \bF_{p^{d_j}})|.\]
\end{lem}

\begin{proof} Let $\tilde{R} := \bZ_{p}[t]/(P(t))$. Choose a short exact sequence
\[0 \ra A \ra R^{s} \overset{\phi}{\longrightarrow} G \ra 0\]
of $R$-modules for some $s \in \bZ_{\geq 1}$, as $G$ is of finite size. We can extend this to the following commutative diagram:
\[\xymatrix{
0 \ar[r] & \tilde{A} \ar[r] \ar[d]^{\hspace{-3mm}\mod p^{k}} & \tilde{R}^s \ar[r] \ar[d]^{\hspace{-3mm}\mod p^{k}} & G \ar[r] \ar[d]^{\mr{id}} & 0 \\
0 \ar[r] & A \ar[r] & R^s \ar[r]^{\phi} & G \ar[r] & 0,
}\]
where the first row is an exact sequence of $\tilde{R}$-modules and
\[\tilde{A} := \{v \in \tilde{R}^{s} : \phi(\bar{v}) = 0\},\]
where $\bar{v} \in R^{s}$ is the image of $v$ modulo $p^{k}$. This induces the following commutative diagram of $\tilde{R}$-modules:
\[\xymatrix{
0 \ar[r] & \Hom_{R}(G, \bF_{p^{d_j}}) \ar[r] \ar[d] & \Hom_{R}(R^{s}, \bF_{p^{d_j}}) \ar[r] \ar[d] & \Hom_{R}(A, \bF_{p^{d_j}}) \ar[r] \ar[d] & \Ext^{1}_{R}(G, \bF_{p^{d_j}}) \ar[r] & 0 \\
0 \ar[r] & \Hom_{\tilde{R}}(G, \bF_{p^{d_j}}) \ar[r] & \Hom_{\tilde{R}}(\tilde{R}^{s}, \bF_{p^{d_j}}) \ar[r] & \Hom_{\tilde{R}}(\tilde{A}, \bF_{p^{d_j}}) \ar[r] & \Ext^{1}_{\tilde{R}}(G, \bF_{p^{d_j}}) \ar[r] & 0,
},\]
where the first row is an exact sequence of $R$-modules, while the second row is an exact sequence of $\tilde{R}$-modules. (This uses that $\Ext^{1}$ over any ring vanishes on free modules.) The first two vertical maps are $\tilde{R}$-linear isomorphisms, and the third vertical map is injective. Given every $\tilde{R}$-linear map $\psi : \tilde{A} \ra \bF_{p^{d_j}}$, we see $\psi(p^{k}v) = p^{k}\psi(v) = 0$ for all $v \in \tilde{A}$ because $p$ annihilates $\bF_{p^{d_{j}}}$. Thus, it follows that $\psi$ must factor as $\tilde{A} \tra A \ra \bF_{p^{d_{j}}}$, which shows that the third vertical map is surjective, and thus bijective. Thus, we have
\begin{align*}
|\Ext^{1}_{\tilde{R}}(G, \bF_{p^{d_j}})| &= |\Hom_{\tilde{R}}(\tilde{A}, \bF_{p^{d_j}})||\Hom_{\tilde{R}}(\tilde{R}^{m}, \bF_{p^{d_j}})|^{-1}|\Hom_{\tilde{R}}(G, \bF_{p^{d_j}})| \\
&= |\Hom_{R}(A, \bF_{p^{d_j}})||\Hom_{R}(R^{m}, \bF_{p^{d_j}})|^{-1}|\Hom_{R}(G, \bF_{p^{d_j}})| \\
&= |\Ext^{1}_{R}(G, \bF_{p^{d_j}})|,
\end{align*}
as desired.
\end{proof}

\begin{lem}\label{extineq} Keeping the notation above, for any $1 \leq j \leq l$, we have 
\[|\Ext^{1}_{\bZ_{p[t]}/(P(t))}(G, \bF_{p^{d_j}})| = p^{d_{j} u}|\Hom_{\bZ_{p}[t]}(G, \bF_{p^{d_j}})|\]
for some $u \in \bZ_{\geq 0}$ so that
\[|\Hom_{\bZ_{p}[t]}(G, \bF_{p^{d_j}})| \leq |\Ext^{1}_{\bZ_{p}[t]/(P(t))}(G, \bF_{p^{d_j}})|,\]
where the equality holds if $m_{j} = 1$ in the factorization (\ref{fac}). In particular, the equality holds when the reduction $\bar{P}(t)$ modulo $p$ of the given monic polynomial $P(t) \in \bZ_{p}[t]$ is square-free in $\bF_{p}[t]$.
\end{lem}

\begin{proof} By Lemma \ref{extred}, fixing $k \in \bZ_{\geq 1}$ such that $p^{k-1}G = 0$, it is enough to show that 
\[p^{d_{j}u}|\Hom_{R}(G, \bF_{p^{d_j}})| = |\Ext^{1}_{R}(G, \bF_{p^{d_j}})|\] for some $u \in \bZ_{\geq 0}$ so that
\[|\Hom_{R}(G, \bF_{p^{d_j}})| \leq |\Ext_{R}(G, \bF_{p^{d_j}})|\]
and that the equality holds when $m_{j} = 1$. By the Chinese Remainder Theorem, we have
\[(\bZ/p^{k}\bZ)[t]/(P(t)) = R \simeq R_{1} \times \cdots \times R_{l},\]
where $R_{i} := (\bZ/p^{k}\bZ)[t]/(Q_{i}(t))$. (The isomorphism above is an isomorphism of $R$-algebras.) We then also have
\[G \simeq_{R} G_{1} \times \cdots \times G_{l},\]
where $G_{i}$ is an $R_{i}$-module. Note that
\bi
	\item $\Hom_R(G, \bF_{p^{d_j}}) \simeq_{R} \Hom_{R_j}(G_j, \bF_{p^{d_j}})$ and
	\item $\mathrm{Ext}_R^1(G, \bF_{p^{d_j}}) \simeq_{R} \mathrm{Ext}_{R_j}^1(G_j, \bF_{p^{d_j}})$
\ei
because $\bF_{p^{d_j}} = \bF_{p}[t]/(\bar{P}_{j}(t))$ is an $R_{j}$-module. Thus, it remains to show that
\[|\Ext^{1}_{R_j}(G_j, \bF_{p^{d_j}})| = p^{d_{j} u}|\Hom_{R_j}(G_j, \bF_{p^{d_j}})|\]
for some $u \in \bZ_{\geq 0}$ and that we necessarily get $u = 0$ when $m_j = 1$. Fix a lift $P_j(t) \in (\bZ/p^{k}\bZ)[t]$ of $\bar{P}_j(t) \in \bF_{p}[t]$. Then we note that $R_j = (\bZ/p^{k}\bZ)[t]/(Q_j(t))$ is a local ring with the maximal ideal $\mf{m}_j = (p, P_j(t))/(Q_j(t))$ and the residue field $R_j/\mf{m}_j \simeq \bF_{p}[t]/(\bar{P}_j(t)) = \bF_{p^{d_j}}$.

\hspace{3mm} Since $G_j$ is of finite size, we may choose a short exact sequence of $R_j$-modules
\begin{equation*}
0 \to A \to R_j^s \to G \to 0
\end{equation*}
for some $s \in \bZ_{\geq 1}$. Using that $\mathrm{Ext}^1_{R_j}(R_j^s, \bF_{p^{d_j}}) = 0$, we derive the following exact sequence of $R_j$-modules:
\[0 \to \Hom_{R_j}(G_j,\bF_{p^{d_j}}) \to \Hom_{R_j}(R_j^s, \bF_{p^{d_j}}) \to \Hom_{R_j}(A, \bF_{p^{d_j}}) \to \mathrm{Ext}_{R_j}^1(G_j, \bF_{p^{d_j}})\to 0,\]
so
\begin{equation}
\label{eq: ext computation}
|\mathrm{Ext}_{R_j}^1(G_j, \bF_{p^{d_j}})| = \frac{|\Hom_{R_j}(A, \bF_{p^{d_j}})|}{|\Hom_{R_j}({R_j}^s, \bF_{p^{d_j}})|}|\Hom_{R_j}(G_j,\bF_{p^{d_j}})|.
\end{equation}
This implies that the following are equivalent:
\bi
	\item $|\Hom_{R_j}(G_j,\bF_{p^{d_j}})| \leq |\mathrm{Ext}_{R_j}^1(G_j, \bF_{p^{d_j}})|$ and
	\item $|\Hom_{R_j}(R_j^s, \bF_{p^{d_j}})| \le |\Hom_{R_j}(A, \bF_{p^{d_j}})|$,
\ei
where the equality on each line holds if and only if the equality on the other line holds.

\hspace{3mm} For any $R_j$-module $M$, we write $M[p] := \{v \in M : pv = 0\}$, which is an $R_j$-submodule of $M$. Since $p^{k-1}G_j = 0$, we have $R_j^s[p] = p^{k-1}R_j^s \sub A \sub R_j^s$. Hence, we have $R_j^s[p] = p^{k-1}R_j^s = A[p]$. (Note that this crucially uses that $p^{k-1}G_j = 0$ instead of $p^{k}G_j = 0$ because otherwise, we do not necessarily get $R_j^s[p] \sub A$.) Then we consider the following commutative diagram:
\begin{equation}
\label{eq: snake}
\xymatrix{
0 \ar[r] & A \ar[r] \ar[d]^p & R_j^s \ar[r] \ar[d]^p & G_j \ar[r] \ar[d]^p & 0 \\
0 \ar[r] & A \ar[r] & R_j^s \ar[r] & G_j \ar[r] & 0,
}
\end{equation}
where the vertical maps are given by multiplication by $p$. By the snake lemma for \eqref{eq: snake} and the above observation that $A[p] = R_j^{s}[p]$, we get the following exact sequence of $R_j$-modules:
\[0 \to G_j[p] \to A/pA \to R_j^s/p R_j^s \to G_j/pG_j \to 0.\]
From the exact sequence
\[0 \ra G_j[p] \ra G_j \overset{p}{\longrightarrow} G_j \ra G_j/pG_j \ra 0,\]
we have $|G_j[p]| = |G_j/pG_j|$, so the previous exact sequence implies that
\begin{equation}
\label{eq: Fp dim}
\dim_{\bF_p} (A/pA) = \dim_{\bF_p} (R_j^s/pR_j^s) = \dim_{\bF_p} (R_j/pR_j)^s = s\dim_{\bF_p} (\bF_{p}[t]/(\bar{P}_j(t)^{m_j}) = sm_jd_j.
\end{equation}
As $\bF_p[t]/(\bar{P}_j(t)^{m_j})$-modules, we may write
\[A/pA \simeq \bF_p[t]/(\bar{P}_j(t)^{e_1}) \times \cdots \times \bF_p[t]/(\bar{P}_j(t)^{e_r}),\]

where $1 \le e_1,\ldots, e_r \le m_j$. (When $A/pA = 0$, we have $r = 0$.) With this decomposition \eqref{eq: Fp dim} says
\[(e_1 + \cdots + e_r)d_j = sm_jd_j,\]
or equivalently $sm_j = e_1 + \cdots + e_r$. In particular, we have
\begin{equation}\label{sm}
sm_j = e_{1} + \cdots + e_{r} \leq m_j + \cdots + m_j = rm_j.
\end{equation}
Hence, we have $s \leq r$. We also note that if $m_j = 1$, then $e_{1} = \cdots = e_{r} = 1 = m_j$, so that the equality is achieved in \eqref{sm} to imply that $s = r$.

\hspace{3mm} We have
\[\Hom_{R_j}(A, \bF_{p^{d_j}}) \simeq_{R} \Hom_{R_j}(\bF_p[t]/(\bar{P}_j(t)^{e_1}), \bF_{p^{d_j}}) \times \cdots \times \Hom_{R_j}(\bF_p[t]/(\bar{P}_j(t)^{e_r}), \bF_{p^{d_j}}),\]
and
\begin{align*}
\Hom_{R_j}(\bF_p[t]/(\bar{P}_j(t)^{e_i}), \bF_{p^{d_j}}) &= \Hom_{R_j}(\bF_p[t]/(\bar{P}_j(t)^{e_i}), \bF_{p}[t]/(\bar{P}_j(t))) \\
&= \Hom_{\bF_{p}[t]}(\bF_p[t]/(\bar{P}_j(t)^{e_i}), \bF_{p}[t]/(\bar{P}_j(t))) \\
&\simeq \Hom_{\bF_{p}[t]}(\bF_p[t]/(\bar{P}_j(t)), \bF_{p}[t]/(\bar{P}_j(t))) \\
&= \Hom_{\bF_{p}[t]/(\bar{P}_j(t))}(\bF_p[t]/(\bar{P}_j(t)), \bF_{p}[t]/(\bar{P}_j(t))) \\
&\simeq \bF_{p}[t]/(\bar{P}_j(t)) \\
&= \bF_{p^{d_j}}
\end{align*}
so that $|\Hom_{R_j}(A, \bF_{p^{d_j}})| = p^{d_j r}$. This implies that
\[|\Hom_{R_j}(R_j^s, \bF_{p^{d_j}})| = p^{d_js} \le p^{d_jr} = |\Hom_{R_j}(A, \bF_{p^{d_j}})|\]
and when $m_j = 1$, the equality is achieved. It also follows from \eqref{eq: ext computation} that
\[|\mathrm{Ext}_{R_j}^1(G_j, \bF_{p^{d_j}})| = p^{d_j(r-s)}|\Hom_{R_j}(G_j,\bF_{p^{d_j}})|\]
This finishes the proof.
\end{proof}

\begin{lem}\label{vanishing} Keeping the notation above, for any $1 \leq j \leq l$, if
\[|\Hom_{\bZ_{p}[t]}(G, \bF_{p^{d_j}})| < |\Ext_{\bZ_{p}[t]/(P(t))}(G, \bF_{p^{d_j}})|,\]
then
\[\displaystyle\prod_{i=1}^{\infty}\lt(1 - \frac{|\Ext_{\bZ_{p}[t]/(P(t))}^{1}(G, \bF_{p^{d_{j}}})| p^{-id_{j}}}{|\Hom_{\bZ_{p}[t]}(G, \bF_{p^{d_{j}}})|}\rt) = 0\]
so that as $n \ra \infty$, the limit of the probability appearing in Theorem \ref{main} is $0$.
\end{lem}

\begin{proof} By Lemma \ref{extineq}, the hypothesis implies that
\[\frac{|\Ext_{\bZ_{p}[t]/(P(t))}^{1}(G, \bF_{p^{d_{j}}})|}{|\Hom_{\bZ_{p}[t]}(G, \bF_{p^{d_{j}}})|} = p^{d_{j}u}\]
for some $u \in \bZ_{\geq 1}$. This implies that
\[\prod_{i=1}^{\infty}\lt(1 - \frac{|\Ext_{\bZ_{p}[t]/(P(t))}^{1}(G, \bF_{p^{d_{j}}})| p^{-id_{j}}}{|\Hom_{\bZ_{p}[t]}(G, \bF_{p^{d_{j}}})|}\rt) = \prod_{i=1}^{\infty}(1 - p^{(u-i)d_{j}}) = 0.\]
\end{proof}

\begin{rmk}[Algorithm for computing $|\Ext_{\bZ_{p}[t]/(P(t))}(G, \bF_{p^{d_j}})|$]\label{algorithm} We note that the proof of Lemma \ref{extineq} gives an algorithm that computes $|\Ext^{1}_{\bZ_{p}/(P(t))}(G, \bF_{p^{d_j}})|$. First, we choose $k \in \bZ_{\geq 1}$ such that $p^{k-1}G = 0$. Write $R := (\bZ/p^{k}\bZ)[t]/(P(t))$, with which we have
\[R \simeq R_{1} \times \cdots \times R_{l}\]
as isomorphism of $R$-algebras where $R_{i} := (\bZ/p^{k}\bZ)[t]/(Q_i(t))$. Then we have
\[G \simeq_{R} G_{1} \times \cdots \times G_{l},\]
where $G_{i}$ is an $R_{i}$-module. Consider a short exact sequence of $R_j$-modules of the form
\[0 \ra A \ra R_j^{s} \ra G_j \ra 0\]
for some $s \in \bZ_{\geq 1}$. Then it is always the case that $\dim_{\bF_{p}}(A/pA) = sm_{j}d_{j}$, and
\begin{align*}
|\Ext_{\bZ_{p}[t]/(P(t))}(G, \bF_{p^{d_j}})| &= \lt|\frac{(A/pA)}{(\bar{P}_{j}(t))(A/pA)}\rt|p^{-d_{j}s} |\Hom_{\bZ_{p}[t]}(G, \bF_{p^{d_j}})| \\
&= \lt|\frac{(A/pA)}{(\bar{P}_{j}(t))(A/pA)}\rt|p^{\dim_{\bF_p}(G/pG) - d_j s}.
\end{align*}
\end{rmk}
 
\begin{exmp}\label{extexmp}
If $\bar{P}(t) \in \bF_{p}[t]$ is not square-free, the inequality in Lemma \ref{extineq} may be a strict inequality. For example, take $P(t) = t^{2}$ and let $G := \bF_{p}[t]/(t) = \bF_{p}$, which we may see as a module over $R = (\bZ/p^2\bZ)[t]/(t^2)$. Consider the modulo $(p, t)$ projection $R \tra G$, and we may form a short exact sequence
\[0 \ra A \ra R \ra G \ra 0\]
of $R$-modules, where $A = (p, t)/(t^2) \sub R$. We have 
\[\dim_{\bF_p} (A/pA) = 1 \cdot 2 \cdot 1 = 2\]
by Remark \ref{algorithm} with $d_{j} = s = 1$ and $m_{j} = 2$. Since $A/pA$ is annihilated by $t$, it follows that
\[A/pA \simeq \bF_{p}[t]/(t) \times \bF_{p}[t]/(t)\]
as $\bF_{p}[t]/(t^{2})$-modules (or as $R$-modules), so
\[ \lt|\frac{(A/pA)}{\bar{t}(A/pA)}\rt|  = p^{2}.\]
It follows from Remark \ref{algorithm} with $d_{j} = s = 1$ that
\[|\mathrm{Ext}_{\bZ_{p}[t]/(P(t))}^1(G, \bF_{p})| = p|\Hom_{\bZ_{p}[t]}(G,\bF_{p})| > |\Hom_{\bZ_{p}[t]}(G,\bF_{p})|.\]
\end{exmp}

\section{Reductions for Theorem \ref{main}}\label{reduction}

\subsection{Theorem \ref{main2} implies Theorem \ref{main}} We start this section with a lemma, mentioned in the introduction, which proves that if we replace $\bZ_{p}$ with $\bZ/p^{k}\bZ$ in Theorem \ref{main} with large enough $k \in \bZ_{\geq 1}$ (or more precisely, any $k \in \bZ_{\geq 1}$ that satisfies $p^{k-1}G = 0$), we get an equivalent statement. We also note that this reduction uses Lemma \ref{extred}.

\begin{lem}\label{red} Let $P(t) \in \bZ_{p}[t]$ be a monic polynomial, and fix a finite size module $G$ over $\bZ_{p}[t]/(P(t))$. Choose any $k \in \bZ_{\geq 1}$ such that $p^{k-1}G = 0$. Then the following are equivalent: 
\bi
	\item $\cok(P(X)) \simeq_{\bZ_{p}[t]} G$;
	\item $\cok(P(X')) \simeq_{(\bZ/p^{k}\bZ)[t]} G$,
\ei
where $X' \in \M_{n}(\bZ/p^{k}\bZ)$ is the image of $X$ modulo $p^{k}$. 
\end{lem}

\begin{proof} If $\cok(P(X)) \simeq_{\bZ_{p}[t]} G$, then 
\begin{align*}
\cok(P(X')) &\simeq_{(\bZ/p^{k}\bZ)[t]} \cok(P(X)) \ot_{\bZ_{p}} (\bZ/p^{k}\bZ) \\
&\simeq_{(\bZ/p^{k}\bZ)[t]} G \ot_{\bZ_{p}} (\bZ/p^{k}\bZ) \\
&\simeq_{(\bZ/p^{k}\bZ)[t]} G.
\end{align*} 
Hence, we assume that $\cok(P(X')) \simeq_{(\bZ/p^{k}\bZ)[t]} G$ and show $\cok(P(X)) \simeq_{\bZ_{p}[t]} G$. Since $\cok(P(X')) \simeq_{(\bZ/p^{k}\bZ)[t]} G$ we have $\cok(P(X')) \simeq_{\bZ/p^{k}\bZ} G$, and thus it follows that $\cok(P(X)) \simeq_{\bZ_{p}} G$, using the Smith normal form of $P(X)$. (See \cite[Lemma 4.1]{CH} for details.) In particular, we have $|\cok(P(X))| = |G| = |\cok(P(X'))|$. This implies that the projection $\cok(P(X)) \tra \cok(P(X'))$ modulo $p^{k}$ is a bijection, and thus it is a $\bZ_{p}[t]$-linear isomorphism. This implies that
\[\cok(P(X)) \simeq_{\bZ_{p}[t]} \cok(P(X')) \simeq_{\bZ_{p}[t]} G\]
since any $(\bZ/p^{k}\bZ)[t]$-linear map is a $\bZ_{p}[t]$-linear map. This finishes the proof.
\end{proof}

\begin{notn} From now on, we fix $k \in \bZ_{\geq 1}$ and a non-constant monic polynomial $P(t) \in (\bZ/p^k\bZ)[t]$. We consider the unique factorization of the reduction $\bar{P}(t)$ of $P(t)$ modulo $p$ as follows: 
\[\bar{P}(t) = \bar{P}_{1}(t)^{m_{1}} \cdots \bar{P}_{l}(t)^{m_{l}},\]
where $\bar{P}_{j}(t) \in \bF_{p}[t]$ are distinct monic irreducible polynomials and $m_{j} \in \bZ_{\geq 1}$. We write $d_{j} := \deg(\bar{P}_{j})$.
\end{notn}

\hspace{3mm} Lemma \ref{extred} and Lemma \ref{red} imply that to prove Theorem \ref{main}, it is enough to prove the following:

\begin{thm}\label{main3} Let $R := (\bZ/p^{k}\bZ)[t]/(P(t))$. For any $\ep$-balanced measures on $(\M_{n}(\bZ/p^k\bZ))_{n \in \bZ_{\geq 1}}$ and any finite size $R$-module $G$, we have
\[\lim_{n \ra \infty}\underset{X \in \M_{n}(\bZ/p^{k}\bZ)}{\Prob}(\cok(P(X)) \simeq_{R} G)
= \dfrac{1}{|\Aut_{R}(G)|}\displaystyle\prod_{j=1}^{l}\displaystyle\prod_{i=1}^{\infty}\lt(1 - \frac{|\Ext_{R}^{1}(G, \bF_{p^{d_{j}}})| p^{-id_{j}}}{|\Hom_{R}(G, \bF_{p^{d_{j}}})|}\rt).\]
\end{thm}

\hspace{3mm} The following is a special case of \cite[Theorem 1.6 and Lemma 6.3]{SW} by taking $R = S = (\bZ/p^{k}\bZ)[t]/(P(t))$ in the cited paper.

\begin{thm}[Sawin and Wood]\label{SW} Let $R := (\bZ/p^{k}\bZ)[t]/(P(t))$, and denote by $\mc{C}$ the set of isomorphism classes of finite size $R$-modules. Let $(U_{n})_{n \geq 1}$ be a sequence of random elements in $\mc{C}$ such that
\[\lim_{n \ra \infty}\underset{U_{n} \in \mc{C}}{\bE}|\Sur_{R}(U_{n}, G)| = 1\]
for every $G$ in $\mc{C}$. Then for any $G \in \mc{C}$, we have
\[\lim_{n \ra \infty}\underset{U_{n} \in \mc{C}}{\Prob}(U_{n} \simeq_{R} G)
= \dfrac{1}{|\Aut_{R}(G)|}\displaystyle\prod_{j=1}^{l}\displaystyle\prod_{i=1}^{\infty}\lt(1 - \frac{|\Ext_{R}^{1}(G, \bF_{p^{d_{j}}})| p^{-id_{j}}}{|\Hom_{R}(G, \bF_{p^{d_{j}}})|}\rt).\]
\end{thm}

\hspace{3mm} Taking $U_{n} = \cok(P(X_{n}))$, where $X_{n}$ is a random matrix in $\M_{n}(\bZ/p^{k}\bZ)$ with the given $\ep$-measure, Theorem \ref{SW} shows that proving Theorem \ref{main2} implies Theorem \ref{main3}, which implies Theorem \ref{main}.

\subsection{A further reduction} Given a finite size $R$-module $G$, writing $\mu_{n}$ to mean the given probability measure on $\M_{n}(\bZ/p^{k}\bZ)$, we have
\begin{align*}
\underset{X \in \M_{n}(\bZ/p^{k}\bZ)}{\bE}|\Sur_{R}(\cok_{R}(X - \bar{t}I_{n}), G)| &= \int_{X \in \M_{n}(\bZ/p^{k}\bZ)} |\Sur_{R}(\cok_{R}(X - \bar{t}I_{n}), G)| d\mu_{n} \\
&=  \int_{X \in \M_{n}(\bZ/p^{k}\bZ)} \sum_{\bar{F} \in \Sur_{R}(\cok_{R}(X - \bar{t}I_{n}), G)}1 d\mu_{n} \\
&=  \sum_{F \in \Sur_{R}(R^{n}, G)} \int_{X \in \M_{n}(\bZ/p^{k}\bZ)}\bb{1}(F(X - \bar{t}I_{n}) = 0) d\mu_{n} \\
&= \sum_{F \in \Sur_{R}(R^{n}, G)} \underset{X \in \M_{n}(\bZ/p^{k}\bZ)}{\Prob}(F(X - \bar{t}I_{n}) = 0),
\end{align*}
where we denote by $\bb{1}(\ms{P})$ the indicator function for a given statement $\ms{P}$. Hence, to prove Theorem \ref{main2}, we study the contribution of 
\[\underset{X \in \M_{n}(\bZ/p^{k}\bZ)}{\Prob}(F(X - \bar{t}I_{n}) = 0).\]
It is extremely important to note that some $F \in \Sur_R(R^n, G)$ forces the above probability to be $0$. That is, in order to satisfy $F(X - \bar{t}I_n) = 0$, we must have $F(\bar{t}(\bZ/p^k\bZ)^n) = F(X(\bZ/p^k\bZ)^n) \sub F((\bZ/p^k\bZ)^n)$ because entries of $X$ are in $\bZ/p^k\bZ$. Since $R^n = (\bZ/p^k\bZ)^n + \bar{t}(\bZ/p^k\bZ)^n + \cdots + \bar{t}^{d-1}(\bZ/p^k\bZ)^n$ with $d = \deg(P)$, this implies that we must have $F((\bZ/p^k\bZ)^n) = F(R^n) = G$. Hence, we shall define
\begin{equation}\label{importantF}
\Sur_R(R^n, G)^{\#} := \{F \in \Sur_R(R^n, G) : F((\bZ/p^k\bZ)^n) = G\}.
\end{equation}

Note that
\begin{align*}
&\lt|\underset{X \in \M_{n}(\bZ/p^{k}\bZ)}{\bE}|\Sur_{R}(\cok_{R}(X - \bar{t}I_{n}), G)| - 1\rt| \\
&= \lt|\sum_{F \in \Sur_{R}(R^{n}, G)} \underset{X \in \M_{n}(\bZ/p^{k}\bZ)}{\Prob}(F(X - \bar{t}I_{n}) = 0) - \sum_{F \in \Hom_{R}(R^{n}, G)}|G|^{-n}\rt| \\
&\leq \sum_{F \in \Sur_{R}(R^{n}, G)^{\#}}\lt| \underset{X \in \M_{n}(\bZ/p^{k}\bZ)}{\Prob}(F(X - \bar{t}I_{n}) = 0) - |G|^{-n}\rt| + \sum_{F \in \Hom_{R}(R^{n}, G) \sm \Sur_{R}(R^{n}, G)^{\#}} |G|^{-n}
\end{align*}
and
\begin{align*}
\sum_{F \in \Hom_{R}(R^{n}, G) \sm \Sur_{R}(R^{n}, G)^{\#}} |G|^{-n} &= \sum_{H \lneq G}\sum_{\substack{F \in\Hom_{R}(R^{n}, G): \\ F(v_1), \dots, F(v_n) \in H}} |G|^{-n} \\
&\leq \sum_{H \lneq G} \lt(\frac{|H|}{|G|}\rt)^{n} \\
&\leq  N_{G} \lt(\frac{1}{2}\rt)^{n}
\end{align*}
where $v_1, \dots, v_n$ are the standard vectors in $(\bZ/p^k\bZ)^n \sub R^n$ and $N_{G}$ is the number of proper subgroups of $G$. The last quantity goes to $0$ as $n \ra \infty$, so to prove Theorem \ref{main2} (and Theorem \ref{main}), it is enough to show the following:

\begin{thm}\label{main4} Let $G$ be any finite size $R$-module. Then
\[\lim_{n \ra \infty}\sum_{F \in \Sur_{R}(R^{n}, G)^{\#}}\lt| \underset{X \in \M_{n}(\bZ/p^{k}\bZ)}{\Prob}(F(X - \bar{t}I_{n}) = 0) - |G|^{-n}\rt| = 0.\]
\end{thm}

\hspace{3mm} The rest of the paper is devoted to showing Theorem \ref{main4}.

\

\section{Application of the discrete Fourier transform for Theorem \ref{main4}}

\subsection{Discrete Fourier transform} We give a way of expressing the indicator function that tells us whether a fixed element of a finite size $R$-module $G$ is trivial or not as an average over an abelian group of size $|G|$, where $R$ is a commutative ring with unity of finite size. We shall see that this average expression, stated as Lemma \ref{DF}, lets us decompose 
\[\underset{X \in \M_{n}(\bZ/p^{k}\bZ)}{\Prob}(F(X - \bar{t}I_{n}) = 0)\]
which appears in Theorem \ref{main4}, when $R = (\bZ/p^{k}\bZ)[t]/(P(t))$. 

\begin{notn} As in the introduction, given a commutative ring $A$ with unity and $A$-modules $U$ and $V$, denote by $\Hom_{A}(U, V)$ the set of $A$-linear maps from $U$ to $V$. We write $\Hom(U, V) := \Hom_{\bZ}(U, V)$.
\end{notn}

\hspace{3mm} The following lemma and its corollary were used in \cite{Woo19}, but we add the proofs of them here for the convenience of the reader.

\begin{lem}\label{count'} Let $A$ be a commutative ring with unity of finite size. For any $a \in A$, we have
\[|\Hom_{A}(A/aA, A)| = |A/aA|.\]
\end{lem}

\begin{proof} There is a bijection between $\Hom_{A}(A/aA, A)$ and $\{\phi \in \Hom_{A}(A, A) : \phi(a) = 0\}$, the latter of which is isomorphic to
\[\Ann(a) = \{r \in A : ar = 0\}\]
by the map $\phi \mapsto \phi(1)$. In particular, we have
\[|\Hom_{A}(A/aA, A)| = |\Ann(a)|.\]
Note that $\Ann(a)$ is the kernel of the surjective map $A \tra aA$ given by $r \mapsto ar$, so we have $A/\Ann(a) \simeq aA$. This implies that $|A|/|\Ann(a)| = |A/\Ann(a)| = |aA|$, so
\[|A/aA| = |A|/|aA| = |\Ann(a)| = |\Hom_{A}(A/aA, A)|,\]
as desired.
\end{proof}

\begin{cor}\label{count} Let $A$ be a commutative ring with unity of finite size, and let $G := A/a_{1}A \op \cdots \op A/a_{l}A$ for some $a_{1}, \dots, a_{l} \in A$. We have
\[|\Hom_{A}(G, A)| = |G|.\]
\end{cor}

\begin{proof} Since
\[\Hom_{A}(A/a_{1}A \op \cdots \op A/a_{l}A, A) \simeq_{A} \Hom_{A}(A/a_{1}A, A) \times \cdots \times \Hom_{A}(A/a_{l}A, A),\]
by Lemma \ref{count'}, we have
\[|\Hom_{A}(G, A)| = |A/a_{1}A| \cdots |A/a_{l}A| = |G|.\]
\end{proof}

\begin{exmp} When $A$ is a finite quotient of a principal ideal domain (PID), then every finite size $A$-module $G$ is of the form in the hypothesis of Corollary \ref{count}, so we have $\Hom_{A}(G, A) = |G|$. In particular, this identity works when $A = \bZ/m\bZ$ for any $m \in \bZ_{\geq 1}$.
\end{exmp}

\hspace{3mm} The following is a formal definition, which is convenient for the proof of Lemma \ref{DF}:

\begin{defn}\label{annrev} Let $A$ be a commutative ring with unity. We say that $A$ is \textbf{annihilator-reversing} if for any $x, y \in A$ with $\Ann(x) \sub \Ann(y)$, we have $xA \sups yA$.
\end{defn}

\begin{exmp} Let $D$ be any PID and fix any nonzero $f \in D$. Here, we show that $A = D/fD$ has the annihilator-reversing property. Fix any $x, y \in A$ such that $\Ann(x) \sub \Ann(y)$. Since $D$ is a PID, we may write $xA = \tilde{x}D/fD$ for some $\tilde{x} \in D$ such that $\tilde{x} | f$. Write $f = \tilde{x}g$ for some $g \in D$. Similarly, we can write $yR = \tilde{y}D/fD$ and $f = \tilde{y}h$ for some $\tilde{y}, h \in D$. The image $\bar{g}$ in $A = D/fD$ of $g$ annihilates $x$, so $\bar{g} \in \Ann(x) \sub \Ann(y)$. This implies that $g\tilde{y} \in fD$, so there must be some $u \in D$ such that $g\tilde{y} = fu = \tilde{x}gu$. This implies that $\tilde{y} =  \tilde{x}u \in \tilde{x}D$, so we must have
\[yA = \tilde{y}D/fD \sub \tilde{x}D/fD = xA,\]
showing that $A = D/fD$ has the annihilator-reversing property.
\end{exmp}

\begin{exmp} Taking $D = \bZ$ in the previous example, for any $m \in \bZ_{\geq 0}$, we note that $\bZ/m\bZ$ has the annihilator-reversing property.
\end{exmp}

\begin{lem}[Discrete Fourier Transform]\label{DF} Let $R$ be a commutative ring with unity of finite size so that there exits $m \in \bZ_{\geq 1}$ such that $mR = 0$. Fix any injective group homomorphism $\ld : \bZ/m\bZ \ra \bC^{\times}$ (e.g., $\ld$ can be given as $x \mapsto e^{2\pi i x/m}$). For any $R$-module $G$ and $g \in G$, we have 
\[\bb{1}(g = 0) = \frac{1}{|G|} \sum_{C \in \Hom(G, \bZ/m\bZ)}\ld(C(g)).\]
\end{lem}

\begin{proof} We have a unique ring homomorphism $\bZ/m\bZ \ra R$ because $mR = 0$. Hence, any $R$-module is also a $(\bZ/m\bZ)$-module. By Corollary \ref{count} with $A = \bZ/m\bZ$, we have
\[|\Hom(G, \bZ/m\bZ)| = |\Hom_{\bZ/m\bZ}(G, \bZ/m\bZ)| = |G|,\]
so the result follows when $g = 0$. Hence, suppose that $g \neq 0$, and we show that the sum on the right-hand side is $0$. We note that having $g \neq 0$ also assumes that $G$ is nontrivial.

\hspace{3mm} Note that $G$ is a nontrivial finite size module over $\bZ/m\bZ$, so there exists an isomorphism
\[G \simeq_{\bZ} \bZ/n_{1}\bZ \op \cdots \op \bZ/n_{l}\bZ\]
for some $n_{1}, \dots, n_{l} \geq 2$ with $l \geq 1$ such that $n_{1}, \dots, n_{l}$ divide $m$. Hence, we may identify $G$ with $\bZ/n_{1}\bZ \op \cdots \op \bZ/n_{l}\bZ$ by assuming that $g = (x_{1}, \dots, x_{l}) \in \bZ/n_{1}\bZ \op \cdots \op \bZ/n_{l}\bZ$. Since $g \neq 0$, we must have $x_{j} \neq 0$ in $\bZ/n_{j}\bZ$ for some $1 \leq j \leq l$. We have
\[\Hom(G, \bZ/m\bZ) \simeq_{\bZ} \Hom(\bZ/n_{1}\bZ, \bZ/m\bZ) \times \cdots \times \Hom(\bZ/n_{l}\bZ, \bZ/m\bZ)\]
given by $C \mapsto (C \circ \iota_{1}, \dots, C \circ \iota_{l})$, where $\iota_{k} : \bZ/n_{k}\bZ \hra  \bZ/n_{1}\bZ \op \cdots \op \bZ/n_{l}\bZ = G$ are the inclusions that come with definition of the direct sum. Writing $(C_{1}, \dots, C_{l}) = (C \circ \iota_{1}, \dots, C \circ \iota_{l})$, we have
\[\ld(C(g)) = \ld(C_{1}(x_{1}) + \cdots + C_{l}(x_{l})) = \ld(C_{1}(x_{1})) \cdots \ld(C_{l}(x_{l})),\]
so
\[\sum_{C \in \Hom(G, \bZ/m\bZ)}\ld(C(g)) = \sum_{C_{1} \in \Hom(\bZ/n_{1}\bZ, \bZ/m\bZ)}\ld(C_{1}(x_{1})) \cdots \sum_{C_{l} \in \Hom(\bZ/n_{l}\bZ, \bZ/m\bZ)}\ld(C_{l}(x_{l})).\]
Thus, it is enough to show that 
\[\sum_{C_{j} \in \Hom(\bZ/n_{j}\bZ, \bZ/m\bZ)}\ld(C_{j}(x_{j})) = 0.\]
Let $\tilde{x}_{j}$ be a lift of $x_{j}$ under the projection $\bZ/m\bZ \tra \bZ/n_{j}\bZ$. Since $x_{j} \neq 0$ in $\bZ/n_{j}\bZ$, we have $\tilde{x}_{j} \notin n_{j}(\bZ/m\bZ)$, and since $\bZ/m\bZ$ has the annihilator-reversing property, this implies that $\Ann_{\bZ/m\bZ}(n_{j}) \not{\sub} \Ann_{\bZ/m\bZ}(\tilde{x}_{j})$. Hence, there is $y \in \bZ/m\bZ$ such that $yn_{j} = 0$ while $y\tilde{x}_{j} \neq 0$. This defines a $(\bZ/n_{j}\bZ)$-linear map $\eta : \bZ/n_{j}\bZ \ra \bZ/m\bZ$ given by $(x \mod n_{j}) \mapsto y \cdot (x \mod m)$. This map is well-defined because $yn_{j} = 0$, and we have $\eta(x_{j}) = y\tilde{x}_{j} \neq 0$ in $\bZ/m\bZ$. Since $\ld$ is injective, we must have $\ld(\eta(x_{j})) \neq 0$ with this specific $\eta$. Thus, the map $\Hom_{\bZ/m\bZ}(\bZ/n_{j}\bZ, \bZ/m\bZ) \ra \bC^{\times}$ given by $\varphi \mapsto \ld(\varphi(x_{j}))$ is a nontrivial group homomorphism. This implies that the last sum is $0$, as desired. 
\end{proof}

\subsection{Decomposition of probability} We start analyzing the probabilities appearing in Theorem \ref{main4}. Let $G$ be a finite size module over $R := (\bZ/p^{k}\bZ)[t]/(P(t))$. By definition of the measure we work with, entries of a random matrix $X$ in $\M_{n}(\bZ/p^{k}\bZ)$ are independent, so in particular, their columns $Xv_{1}, \dots, Xv_{n}$ are independent, where we denote by $v_{1}, \dots, v_{n}$ the standard $R$-basis of $R^{n}$. This implies that for any $F \in \Hom_{R}(R^{n}, G)$, we have
\begin{align*}
\underset{X \in \M_{n}(\bZ/p^{k}\bZ)}{\Prob}(F(X - \bar{t}I_{n}) = 0) &= \prod_{j=1}^{n}\underset{X \in \M_{n}(\bZ/p^{k}\bZ)}{\Prob}(F(X - \bar{t}I_{n})v_{j} = 0) \\
&= \prod_{j=1}^{n}\underset{w \in (\bZ/p^{k}\bZ)^{n}}{\Prob}(F(w)  = \bar{t}F(v_{j})).
\end{align*}
Hence, understanding each
\[\underset{w \in (\bZ/p^{k}\bZ)^{n}}{\Prob}(F(w)  = \bar{t}F(v_{j}))\]
is inevitable in proving Theorem \ref{main4}. The following lemma tells us how to decompose this probability using Lemma \ref{DF}. Write $\ze := e^{2\pi i / p^{k}}$, a primitive $p^{k}$-th root of unity.

\begin{lem}\label{prob} For any $F \in \Hom_{R}(R^{n}, G)$ and $h \in G$, we have
\[\underset{w \in (\bZ/p^{k}\bZ)^{n}}{\Prob}(F(w) = h) =  \frac{1}{|G|} \sum_{C \in \Hom(G, \bZ/p^{k}\bZ)}\ze^{-C(h)}\underset{w \in (\bZ/p^{k}\bZ)^{n}}{\bE}(\ze^{C(F(w))}).\]
\end{lem}

\begin{proof} Note that $p^{k}G = 0$ since $G$ is a module over $R = (\bZ/p^{k}\bZ)[t]/(P(t))$, which is annihilated by $p^{k}$. The map $\ld : \bZ/p^{k}\bZ \ra \bC^{\times}$ defined by $\ld(x) := \ze^{x}$ is an injective group homomorphism, so applying Lemma \ref{DF} with $m = p^{k}$, for any $g \in G$, we have
\[\bb{1}(g = 0) = \frac{1}{|G|} \sum_{C \in \Hom(G, \bZ/p^{k}\bZ)}\ze^{C(g)}.\]
Denoting by $v_{1}, \dots, v_{n}$ the standard $R$-basis of $R^{n}$, taking $g = F(w) - h$ for any $w \in R^{n}$ above gives
\[\bb{1}(F(w) = h) = \frac{1}{|G|} \sum_{C \in \Hom(G, \bZ/p^{k}\bZ)}\ze^{-C(h)}\ze^{C(F(w))},\]
so taking the expectation over $w \in (\bZ/p^{k}\bZ)^{n}$ implies that
\[\underset{w \in (\bZ/p^{k}\bZ)^{n}}{\Prob}(F(w)  = h)  = \frac{1}{|G|} \sum_{C \in \Hom(G, \bZ/p^{k}\bZ)}\ze^{-C(h)}\underset{w \in (\bZ/p^{k}\bZ)^{n}}{\bE}(\ze^{C(F(w))}),\]
as desired.
\end{proof}

\begin{cor}\label{sum} For any $F, \varphi \in \Hom_{R}(R^{n}, G)$, we have
\begin{align*}
&\underset{X \in \M_{n}(\bZ/p^{k}\bZ)}{\Prob}(F(X - \bar{t}I_{n}) = \varphi) - |G|^{-n} \\
&= \frac{1}{|G|^{n}} \sum_{\substack{\bs{C} = (C_{1}, \dots, C_{n}) \in \Hom(G, \bZ/p^{k}\bZ)^{n}: \\ \bs{C} \neq (0, \dots, 0)}}\prod_{j=1}^{n}\lt(\ze^{-C_{j}(\bar{t} F(v_{j}) + \varphi (v_{j})) }\underset{w_{j} \in (\bZ/p^{k}\bZ)^{n}}{\bE}(\ze^{C_{j}(F(w_{j}))})\rt).
\end{align*}
\end{cor}

\begin{proof} Lemma \ref{prob} implies that
\begin{align*}
\underset{X \in \M_{n}(\bZ/p^{k}\bZ)}{\Prob}(F(X - \bar{t}I_{n}) = \varphi) &= \prod_{j=1}^{n}\underset{w_{j} \in (\bZ/p^{k}\bZ)^{n}}{\Prob}(F(w_{j})  = \bar{t}F(v_{j}) + \varphi(v_{j}))  \\
&= \frac{1}{|G|^{n}} \prod_{j=1}^{n} \lt( \sum_{C_{j} \in \Hom(G, \bZ/p^{k}\bZ)}\ze^{-C_{j}(\bar{t} F(v_{j}) + \varphi (v_{j}))}\underset{w_{j} \in (\bZ/p^{k}\bZ)^{n}}{\bE}(\ze^{C_{j}(F(w_{j}))}) \rt) \\
&= \frac{1}{|G|^{n}} \sum_{\bs{C} = (C_{1}, \dots, C_{n}) \in \Hom(G, \bZ/p^{k}\bZ)^n}\prod_{j=1}^{n} \lt(\ze^{-C_{j}(\bar{t} F(v_{j}) + \varphi (v_{j}))}\underset{w_{j} \in (\bZ/p^{k}\bZ)^{n}}{\bE}(\ze^{C_{j}(F(w_{j}))}) \rt),
\end{align*}
so considering the summand that corresponds to $C_{1} = \cdots = C_{n} = 0$ separately, we get the desired result.
\end{proof}

\begin{rmk} In Theorem \ref{main4}, we only need to bound the left-hand side of Corollary \ref{sum} when $\varphi = 0$. For certain $F \in \Sur_{R}(R^{n}, G)$, we are able to do this directly, but for other $F$, we have to deal with the left-hand side of Corollary \ref{sum} even when $\varphi$ is nonzero.
\end{rmk}

\hspace{3mm} The following Lemma due to Wood \cite[Lemma 4.2]{Woo17} is extremely important in bounding the summands appearing in Corollary \ref{sum}: 

\begin{lem}\label{rou} For any nonzero $u \in \bZ/p^{k}\bZ$, given any $\ep$-balanced measure on $\bZ/p^{k}\bZ$, we have
\[\lt|\underset{x \in \bZ/p^{k}\bZ}{\bE}(\ze^{ux})\rt| \leq e^{-\ep/p^{2k}}.\]
\end{lem}

\hspace{3mm} Given a random vector $w_{j} = X_{1j}v_{1} + \cdots + X_{nj}v_{n} \in (\bZ/p^{k}\bZ)^{n}$ (which we implicitly think as the $j$-th column of random $X \in \M_{n}(\bZ/p^{k}\bZ)$), according to definition of the given $\ep$-balanced measure, the coefficients $X_{1j}, \dots, X_{nj} \in \bZ/p^{k}\bZ$ are independent. Hence, for any $C_{j} \in \Hom(G, \bZ/p^{k}\bZ)$, we have
\[\underset{w_{j} \in (\bZ/p^{k}\bZ)^{n}}{\bE}(\ze^{C_{j}(F(w_j))}) = \prod_{i=1}^{n}\underset{X_{ij} \in (\bZ/p^{k}\bZ)}{\bE}(\ze^{C_{j}(F(v_{i}))X_{ij}}).\]
Lemma \ref{rou} tells us that whenever $C_{j} \in \Hom(G, \bZ/p^{k}\bZ)$ has many $i \in [n] := \{1, 2, \dots, n\}$ such that $C_{j}(F(v_{i})) \neq 0$, the summand given in Corollary \ref{sum} is negligible. Given any constant $\dt > 0$, for certain kinds of $F \in \Hom_{R}(R^{n}, G)$, we can guarantee that the number of nonzero $i \in [n]$ such that $C_{j}(F(v_{i}))$ is at least $\dt n$, as long as $C_{j} \neq 0$. Such $F$ are called ``codes of distance $\dt n$'' \cite[p.929]{Woo17}, which we explain in the next section.

\

\section{Codes of distance $\dt n$} 

\hspace{3mm} We continue to fix $R := (\bZ/p^{k}\bZ)[t]/(P(t))$ and a finite size $R$-module $G$. In this section, we bound
\[\lt| \underset{X \in \M_{n}(\bZ/p^{k}\bZ)}{\Prob}(F(X - \bar{t}I_{n}) = 0) - |G|^{-n}\rt|\]
appearing as a summand in Theorem \ref{main4} for specific kinds of $F \in \Sur_{R}(R^{n}, G)^{\#}$ called ``codes of distance $\dt n$'' for a fixed constant $\dt > 0$.

\subsection{Preliminaries} We start by introducing convenient notation we use for the rest of the paper.

\begin{notn}\label{prime} We write $V := R^{n} = Rv_{1} \op \cdots \op Rv_{n}$ and $V' := (\bZ/p^{k}\bZ)^{n} = (\bZ/p^{k}\bZ) v_{1} \op \cdots \op (\bZ/p^{k}\bZ) v_{n}$, where $v_{1}, \dots, v_{n}$ form a standard basis both for $V$ over $R$ and for $V'$ over $\bZ/p^{k}\bZ$. Note that for any $F \in \Sur_R(V, G)$, we have $F(V') = F(V)$ if and only if $F \in \Sur_R(V,G)^{\#}$ in the notation \eqref{importantF}. Given $\sg \sub [n]$, we write $V_{\sg}$ to mean the $R$-submodule of $V$ generated by $\{v_{i} : i \in \sg\}$. Similarly, we write $V'_{\sg}$ to mean the $(\bZ/p^{k}\bZ)$-submodule of $V'$ generated by $\{v_{i} : i \in \sg\}$. 
\end{notn}

\begin{rmk} We note that the notation $V_{\sg}$ in \cite{Woo17} means $V_{[n] \sm \sg}$ in this paper. In \cite{Woo19}, the notation $V_{\sm\sg}$ is used to mean $V_{[n] \sm \sg}$ in this paper.
\end{rmk}

\begin{defn} Fix a real number $w > 0$ and a finite size module $M$ over $\bZ/p^k\bZ$. We say $\phi \in \Hom(V', M)$ is a \textbf{code of distance $w$} if for every $\sigma \sub [n]$ with $|\sigma| < w$, we have $\phi(V'_{[n] \sm \sg}) = M$. We say that $F \in \Hom_R(V, G)$ is a \textbf{code of distance $w$} if $F|_{V'} \in \Hom(V', G)$ is a code of distance $w$.
\end{defn}

\begin{rmk} Note that a code $F \in \Hom_{R}(V, G)$ of any distance $w > 0$ satisfies $F(V') = G$. This means that $F$ is not only surjective but also a member of $\Sur_R(V, G)^{\#}$. 
\end{rmk}

\begin{lem}\label{cd} Fix any real number $\dt > 0$. Let $F \in \Hom_{R}(V, G)$ be a code of distance $\dt n$. For any nonzero $C \in \Hom(G, \bZ/p^{k}\bZ)$, we have 
\[\#\{i \in [n] : C(F(v_{i})) \neq 0\} \geq \dt n.\]
\end{lem}

\begin{proof} Since $C \neq 0$ as a map from $G$ to $\bZ/p^{k}\bZ$, we have $\ker(C) \subsetneq G$. If the desired conclusion is false, then there is $\sg \sub [n]$ with $|\sg| < \dt n$ such that every $i \in [n] \sm \sg$ satisfies $C(F(v_{i})) = 0$. In other words, we have $F(V'_{[n] \sm \sg}) \sub \ker(C) \subsetneq G$, contradicting that $F$ is a code of distance $\dt n$. This finishes the proof.
\end{proof}

\begin{cor}\label{naivebd}  Fix any real number $\dt > 0$. If $F \in \Hom_{R}(V, G)$ is a code of distance $\dt n$, then for any nonzero $C_{j} \in \Hom(G, \bZ/p^{k}\bZ)$, we have
\[\lt|\underset{w_{j} \in (\bZ/p^{k}\bZ)^{n}}{\bE}(\ze^{C_{j}(F(w_{j}))})\rt| \leq e^{-\dt\ep n/p^{2k}}.\]
\end{cor}

\begin{proof} Since
\[\underset{w_{j} \in (\bZ/p^{k}\bZ)^{n}}{\bE}(\ze^{C_{j}(F(w_{j}))}) = \prod_{i=1}^{n}\underset{X_{ij} \in (\bZ/p^{k}\bZ)}{\bE}(\ze^{C_{j}(F(v_{i}))X_{ij}}),\]
we get the conclusion by applying Lemma \ref{cd} and Lemma \ref{rou}.
\end{proof}

\subsection{Bounds for codes of distance $\dt n$} Before we give statements and proofs, we explain our strategy, which we mimic from \cite[Lemma 4.1]{Woo17}. Recall from Corollary \ref{sum} with $\varphi = 0$ that
\[\underset{X \in \M_{n}(\bZ/p^{k}\bZ)}{\Prob}(F(X - \bar{t}I_{n}) = 0) - |G|^{-n} = \frac{1}{|G|^{n}} \sum_{\substack{\bs{C} = (C_{1}, \dots, C_{n}) \in \Hom(G, \bZ/p^{k}\bZ)^{n}: \\ \bs{C} \neq (0, \dots, 0)}}\prod_{j=1}^{n}\lt(\ze^{-C_{j}(\bar{t} F(v_{j}))}\underset{w_{j} \in (\bZ/p^{k}\bZ)^{n}}{\bE}(\ze^{C_{j}(F(w_{j}))})\rt).\]
The size of the left-hand side of the above identity is a summand appearing in Theorem \ref{main4}. When we use the bound given by Corollary \ref{naivebd} to bound such a summand for a code $F \in \Hom_{R}(V, R)$ of distance $\dt n$, we can expect that $\bs{C} = (C_{1}, \dots, C_{n})$ with enough nonzero $C_{j}$ would not have a significant contribution. Hence, a natural strategy is to show that many $\bs{C}$ have enough nonzero $C_{j}$ and only few $\bs{C}$ do not have enough nonzero $C_{j}$. This strategy is incorporated in the proof of the following lemma.

\begin{lem}\label{cdbd1}  Fix any real number $\dt > 0$ and $\varphi \in \Hom_{R}(V, G)$. For any real number $0 < \ga < 1$ and a code $F \in \Hom_{R}(V, G)$ of distance $\dt n$, we have
\[\lt|\underset{X \in \M_{n}(\bZ/p^{k}\bZ)}{\Prob}(F(X - \bar{t}I_{n}) = \varphi) - |G|^{-n}\rt| \leq |G|^{(\ga-1)n + 1} {n \choose \lceil \ga n \rceil} e^{-\dt \ep n /p^{2k}} + e^{-\dt\ga\ep n^{2}/p^{2k}}.\]
In particular, we have
\[\sum_{\substack{F \in \Sur_{R}(V, G) : \\ F \text{ code of distance }\dt n}}\lt|\underset{X \in \M_{n}(\bZ/p^{k}\bZ)}{\Prob}(F(X - \bar{t}I_{n}) = 0) - |G|^{-n}\rt| \leq |G|^{\ga n+1} {n \choose \lceil \ga n \rceil} e^{-\dt \ep n /p^{2k}} + |G|^{n}e^{-\dt\ga\ep n^{2}/p^{2k}}.\]
\end{lem}

\begin{proof} Consider
\[\mc{W}_{\ga} := \{\bs{C} = (C_{1} ,\dots, C_{n}) \in \Hom(G, \bZ/p^{k}\bZ)^{n} \sm \{(0, \dots, 0)\} : \#\{j \in [n] : C_{j} \neq 0\} < \ga n\}.\]
(Intuitively, the set $\mc{W}_{\ga}$ consists of $\bs{C}$ that do not have enough nonzero $C_{j}$.) The defining condition of $\mc{W}_{\ga}$ is equivalent to $\#\{j \in [n] : C_{j} = 0\} \geq n - \lceil \ga n \rceil$, so we can bound
\[|\mc{W}_{\ga}| \leq {n \choose n - \lceil \ga n \rceil}|G|^{\lceil \ga n \rceil} = {n \choose \lceil \ga n \rceil}|G|^{\lceil \ga n \rceil} \leq {n \choose \lceil \ga n \rceil}|G|^{\ga n + 1}\]
by selecting $n - \lceil \ga n \rceil$ components $C_{j}$ of $\bs{C}$ that are zeros, and then letting the rest of them free.

\hspace{3mm} Denote by $\mc{R}_{\ga}$ the complement of $\mc{W}_{\ga}$ in $\Hom(G, \bZ/p^{k}\bZ)^{n} \sm \{(0, \dots, 0)\}$. If $\bs{C} \in (C_{1}, \dots, C_{n}) \in \mc{R}_{\ga}$, we have $\#\{j \in [n] : C_{j} \neq 0\} \geq \ga n$, so
\[\prod_{j=1}^{n}\lt|\underset{w_{j} \in (\bZ/p^{k}\bZ)^{n}}{\bE}(\ze^{C_{j}(F(w_{j}))})\rt| \leq e^{-\dt\ga\ep n^{2}/p^{2k}}\]
by Corollary \ref{naivebd}. (Intuitively, the set $\mc{R}_{\ga}$ consists of $\bs{C}$ that have enough nonzero $C_{j}$.) Thus, applying Corollary \ref{sum} and Corollary \ref{naivebd}, we have
\begin{align*}
&\lt|\underset{X \in \M_{n}(\bZ/p^{k}\bZ)}{\Prob}(F(X - \bar{t}I_{n}) = \varphi) - |G|^{-n}\rt| \\
&\leq \frac{1}{|G|^{n}} \sum_{\substack{\bs{C} = (C_{1}, \dots, C_{n}) \in \Hom(G, \bZ/p^{k}\bZ)^{n}: \\ \bs{C} \neq (0, \dots, 0)}}\prod_{j=1}^{n}\lt|\underset{w_{j} \in (\bZ/p^{k}\bZ)^{n}}{\bE}(\ze^{C_{j}(F(w_{j}))})\rt| \\
&\leq |G|^{-n} \lt(\sum_{\bs{C} = (C_{1}, \dots, C_{n}) \in \mc{W}_{\ga}}\prod_{j=1}^{n}\lt|\underset{w_{j} \in (\bZ/p^{k}\bZ)^{n}}{\bE}(\ze^{C_{j}(F(w_{j}))})\rt|\rt) + |G|^{-n} \lt(\sum_{\bs{C} = (C_{1}, \dots, C_{n}) \in \mc{R}_{\ga}}\prod_{j=1}^{n}\lt|\underset{w_{j} \in (\bZ/p^{k}\bZ)^{n}}{\bE}(\ze^{C_{j}(F(w_{j}))})\rt|\rt) \\
&\leq |G|^{-n} {n \choose \lceil \ga n \rceil}|G|^{\ga n + 1} e^{-\dt \ep n /p^{2k}} + |G|^{-n} |G|^{n}e^{-\dt\ga\ep n^{2}/p^{2k}} \\
&= |G|^{(\ga-1)n + 1} {n \choose \lceil \ga n \rceil} e^{-\dt \ep n /p^{2k}} + e^{-\dt\ga\ep n^{2}/p^{2k}},
\end{align*}
as desired.
\end{proof}

\hspace{3mm} To show Theorem \ref{main4}, we want the sum of the right-hand side in Lemma \ref{cdbd1} over codes $F \in \Sur_{R}(V, G)$ of distance $\dt n$ to go to $0$ as $n \ra \infty$. For this, we need to bound the term ${n \choose \lceil \ga n \rceil}|G|^{\ga n}$. Wood's observation, originating from \cite[Lemma 4.1]{Woo17}, is that we can give such a bound by choosing $\ga$ to be small for any large $n$ compared to $\ga$. (This dependence can be neglected because we can fix $\ga$ and let $n \ra \infty$.) This idea is incorporated in the proof of the following lemma.

\begin{lem}[Bounds for codes of distance $\dt n$]\label{cdbd2}  Fix any real number $\dt > 0$ and $\varphi \in \Hom_{R}(V, G)$. Then there exists a real number $0 < \ga < 1$ such that 
\be
	\item for any $n \geq 1/\ga$ and any code $F \in \Hom_{R}(V, G)$ of distance $\dt n$, we have
\[\lt|\underset{X \in \M_{n}(\bZ/p^{k}\bZ)}{\Prob}(F(X - \bar{t}I_{n}) = \varphi) - |G|^{-n}\rt| \leq |G|^{1-n}  \exp\lt(-\frac{\dt \ep n }{4p^{2k}}\rt) + \exp\lt(-\frac{\dt\ga\ep n^{2}}{p^{2k}}\rt),\]
and
	\item there exists $K_{G,\ga} > 0$ that depends on $G$ and $\ga$ (but not on $n$) such that
\[\lt|\underset{X \in \M_{n}(\bZ/p^{k}\bZ)}{\Prob}(F(X - \bar{t}I_{n}) = \varphi) - |G|^{-n}\rt| \leq K_{G,\ga}|G|^{-n}\]
for all $n \in \bZ_{\geq 1}$
\ee
In particular, for $n \geq 1/\ga$, we have
\[\sum_{\substack{F \in \Sur_{R}(V, G)^{\#} : \\ F \text{ code of distance }\dt n}}\lt|\underset{X \in \M_{n}(\bZ/p^{k}\bZ)}{\Prob}(F(X - \bar{t}I_{n}) = 0) - |G|^{-n}\rt| \leq |G|  \exp\lt(-\frac{\dt \ep n }{4p^{2k}}\rt) + |G|^{n}\exp\lt(-\frac{\dt\ga\ep n^{2}}{p^{2k}}\rt)\]
so that
\[\lim_{n \ra \infty}\sum_{\substack{F \in \Sur_{R}(V, G)^{\#} : \\ F \text{ code of distance }\dt n}}\lt|\underset{X \in \M_{n}(\bZ/p^{k}\bZ)}{\Prob}(F(X - \bar{t}I_{n}) = 0) - |G|^{-n}\rt| = 0.\]
\end{lem}

\begin{proof} We first consider (1). Considering Lemma \ref{cdbd1}, it is enough to show that there exits $0 < \ga < 1$ such that
\[|G|^{\ga n} \leq \exp\lt(\frac{\dt \ep n}{2p^{2k}} \rt)\]
and
\[{n \choose \lceil \ga n \rceil} \leq \exp\lt(\frac{\dt \ep n}{4p^{2k}} \rt)\]
for any $n \geq 1/\ga$. Since $|G|^{\ga n} = \exp(n\ga\log(|G|))$, the first inequality is easily achieved, so we focus on the second one. For the second inequality, we use the well-known inequality (so called the \textbf{binary entropy} bound)
\[{n \choose k} \leq e^{n H(k/n)},\]
which holds for integers $0 < k < n$, where $H(\al) := -\al \log(\al) - (1 - \al)\log(1 - \al)$ is defined for real $\al \in (0, 1)$. (A proof can be found in \cite[Example 11.1.3]{CT} replacing $2$ in the reference with $e$.) Hence, for any $0 < \ga < 1$, we have
\[{n \choose \lceil \ga n \rceil} \leq \exp\lt(nH\lt( \frac{\lceil \ga n \rceil}{n} \rt)\rt)\]
Note that $\lim_{\al \ra 0+}H(\al) = 0$ and $0 \leq \lceil \ga n \rceil/n \leq 2\ga$ for all $n \geq 1/\ga$, so we may choose $0 < \ga < 1$ small enough so that
\[H\lt( \frac{\lceil \ga n \rceil}{n} \rt) \leq \frac{\dt \ep}{4p^{2k}}\]
that holds for all $n \geq 1/\ga$. This finishes the proof of (1).

\hspace{3mm} For (2), we take the constant $K_{G, \ga}$ by first considering finitely many $n < 1/\ga$ and then apply (1). 
\end{proof}

\section{Non-codes} 

\subsection{Strategy} Now that we have established Lemma \ref{cdbd2}, to prove Theorem \ref{main4}, it only remains to show that
\begin{equation}\label{last}
\lim_{n \ra \infty}\sum_{\substack{F \in \Sur_{R}(V, G)^{\#} : \\ F \text{ not code of distance }\dt n}}\lt|\underset{X \in \M_{n}(\bZ/p^{k}\bZ)}{\Prob}(F(X - \bar{t}I_{n}) = 0) - |G|^{-n}\rt| = 0
\end{equation}
with a suitable choice of $\dt > 0$.

\hspace{3mm} Given any $\dt > 0$, we shall consider a subset of $\Sur_{R}(V, G)^{\#}$, which contains all $F \in \Sur_{R}(V, G)^{\#}$ that are not codes of distance $\dt n$. (From Definition \ref{depth}, such $F$ are said to be ``of $\dt$-depth $> 1$,'' and this notion was first introduced in \cite[p.936]{Woo17}.) We give an upper bound of the number of such $F$ in Lemma \ref{depths}, which is an analogue of \cite[Lemma 5.2]{Woo17}. Then we give an upper bound for
\[\underset{X \in \M_{n}(\bZ/p^{k}\bZ)}{\Prob}(F(X - \bar{t}I_{n}) = 0)\]
for such $F$ in Lemma \ref{dpbd}. 

\subsection{$\dt$-depth} To define the notion of $\dt$-depth for fixed $\dt > 0$, we first consider the following notation.

\begin{notn} Given $D \in \bZ_{\geq 1}$ with prime factorization $p_{1}^{e_{1}} \cdots p_{r}^{e_{r}}$, with distinct primes $p_{1}, \dots, p_{r}$ and $e_{1}, \dots, e_{r} \in \bZ_{\geq 1}$, where $r \in \bZ_{\geq 0}$, we write $\ell(D) := e_{1} + \cdots + e_{r}$. (Note that $\ell(1) = 0$.)
\end{notn}

\begin{defn}\label{depth} Fix a real number $\dt > 0$ and a finite size module $M$ over $\bZ/p^k\bZ$. Given $\phi \in \Hom(V', M)$, the \textbf{$\dt$-depth} of $\phi$ is the maximal $D \in \bZ_{\geq 1}$ such that there is $\sg \sub [n]$ with $|\sg| < \ell(D) \dt n$ such that $D = [M : \phi(V'_{[n] \sm \sg})]$, with one exception that it is defined to be $1$ if there is no such $D$. We say $F \in \Hom_R(V, G)$ is \textbf{of $\dt$-depth $D$} if $F|_{V'} \in \Hom(V', G)$ is of $\dt$-depth $D$.
\end{defn}

\begin{rmk} Given $\dt > 0$, consider any $F \in \Hom_{R}(V, G)$ with $\dt$-depth $1$. Then for every $\sg \sub [n]$ with $|\sg| < \dt n$, we have $[G : F(V'_{[n] \sm \sg})] = 1$ so that $F(V'_{[n] \sm \sg}) = G$. That is, we see that such an $F$ is necessarily a code of distance $\dt n$. Hence, to consider $F$ that are not code of distance $\dt n$, it suffices to study the ones with $\dt$-depth $> 1$.
\end{rmk}

\begin{rmk}\label{redef} Fix $\dt > 0$ and $D \in \bZ_{> 1}$. If $F \in \Hom_{R}(V, G)$ of $\dt$-depth $D$ so that there exists $\sg \sub [n]$ with $|\sg| < \ell(D)\dt n$ and $D = [G : F(V'_{[n] \sm \sg})]$, then for any $\tau \sub [n]$ that contains $\sg$ with $|\tau| = \lceil \ell(D) \dt n \rceil - 1$, we have 
\[D = [G : F(V'_{[n] \sm \sg})] =  \frac{[G : F(V'_{[n] \sm \tau})]}{[F(V'_{[n] \sm \sg}) : F(V'_{[n] \sm \tau})]},\]
which divides $D' := [G : F(V'_{[n] \sm \tau})]$. This implies that $\ell(D) \leq \ell(D')$, so $|\tau| < \ell(D) \dt n \leq \ell(D') \dt n$. Hence, by the maximality of $D$ in Definition \ref{depth}, it must follow that $D = D'$. 

\hspace{3mm} That is, given integer $D > 1$, if $F \in \Hom_{R}(V, G)$ is of $\dt$-depth $D$, then there exists $\sg \sub [n]$ with $|\sg| = \lceil \ell(D) \dt n \rceil - 1$ such that $D = [G : F(V_{[n] \sm \sg})]$. This observation is useful in estimating the number of $F \in \Hom_{R}(V, G)$ with a fixed $\dt$-depth $> 1$ in the proof of Lemma \ref{depths}.
\end{rmk}

\begin{rmk} The notion of $\dt$-depth is due to Wood \cite[p.936]{Woo17}, but our situation is different because it discusses $R$-modules and $R$-linear maps, and as an abelian group, we have $R \simeq (\bZ/p^{k}\bZ)^{\deg(P)}$ which is not a cyclic group when $\deg(P) > 1$. (In \cite{Woo17}, the notion was defined for $\bZ/a\bZ$ with $a \in \bZ_{\geq 1}$, in place or $R$.) As we discussed before \eqref{importantF}, our key strategy is to note that we may ignore any $F \in \Sur_R(V, G)$ such that $F(V') \neq F(V)$, which we did not need to worry about in \cite{Woo17, Woo19}.
\end{rmk}

\hspace{3mm} The following lemma, which is an analogue of \cite[Lemma 5.2]{Woo17}, gives a useful upper bound to the number of $F \in \Hom_{R}(V, G)$ of a fixed $\dt$-depth $> 1$.

\begin{lem}\label{depths}  Let $G$ be a finite size $R$-module and fix a real number $\dt > 0$. There is $C_{G} > 0$ only depending on $G$ such that for any integer $D > 1$, the number of $F \in \Hom_{R}(V, G)$ of $\dt$-depth $D$ is at most
\[C_{G}{n \choose \lceil \ell(D)\dt n \rceil - 1 }|G|^{n}D^{-n + \ell(D) \dt n}.\]
\end{lem}

\begin{proof} We follow the proof of \cite[Lemma 5.2]{Woo17}. Considering Remark \ref{redef}, the desired number is bounded above by
\[\sum_{\substack{\sg \sub [n]: \\ |\sg| = \lceil \ell(D)\dt n \rceil - 1}}\sum_{\substack{H \leqslant G \text{ subgroup}: \\ [G:H] = D}}\#\{F \in \Hom_{R}(V, G) : F(V'_{[n]\sm\sg}) = H\},\]
which is bounded above by

\begin{align*}
&\sum_{\substack{\sg \sub [n]: \\ |\sg| = \lceil \ell(D)\dt n \rceil - 1}}\sum_{\substack{H \leqslant G \text{ subgroup}: \\ [G:H] = D}}\#\{F \in \Hom_{R}(V, G) : F(V'_{[n]\sm\sg}) \sub H\} \\
&= {n \choose \lceil \ell(D)\dt n \rceil - 1 } \sum_{\substack{H \leqslant G \text{ subgroup}: \\ [G:H] = D}} |H|^{n - \lceil \ell(D)\dt n \rceil + 1}|G|^{\lceil \ell(D)\dt n \rceil - 1} \\
&= {n \choose \lceil \ell(D)\dt n \rceil - 1 } \sum_{\substack{H \leqslant G \text{ subgroup}: \\ [G:H] = D}} |G|^{n}|H|^{n - \lceil \ell(D)\dt n \rceil + 1}|G|^{-n + \lceil \ell(D)\dt n \rceil - 1} \\
&= C_{G, D} {n \choose \lceil \ell(D)\dt n \rceil - 1 } |G|^{n}D^{-n + \lceil \ell(D)\dt n \rceil - 1},
\end{align*}
where $C_{G, D}$ is the number of subgroups $H \le G$ with $[G : H] = D$. Hence, the result follows by taking $C_{G}$ to be the number of all subgroups $H \le G$.
\end{proof}

\begin{rmk} We note that the proof of Lemma \ref{depths} barely needed any change from that of \cite[Lemma 5.2]{Woo17}. The only difference is that we are working with $R$-linear maps whose restrictions give the maps discussed in the cited lemma.
\end{rmk}

\hspace{3mm} We now introduce a lemma that bounds the probability appearing in \eqref{last}. This is where our observation made before \eqref{importantF} is important. That is, in the proof, the condition that $F(V') = F(V) = G$ is used.

\begin{lem}\label{dpbd} Let $G$ be a finite size $R$-module. Fix any real number $\dt > 0$. There exists $K_{G, \dt} > 0$ that only depends on $G$ and $\dt$ (but not depending on $n$) such that for any $F \in \Sur_{R}(V, G)^{\#}$ of $\dt$-depth $D > 1$, we have
\[\underset{X \in \M_{n}(\bZ/p^{k}\bZ)}{\Prob}(F(X - \bar{t}I_{n}) = 0) \leq K_{G,\dt}e^{-\ep n}D^{n}|G|^{-n}.\]
\end{lem}

\begin{proof} Since $F$ has $\dt$-depth $D > 1$, there exists $\sg \sub [n]$ with $|\sg| < \ell(D)\dt n$ such that $D = [G : F(V'_{[n] \sm \sg})]$. Since $F(V') = G$, we have $[G : F(V')] =1 < D = [G : F(V'_{[n] \sm \sg})]$, so $\sg$ is nonempty. Write $H := F(V'_{[n] \sm \sg})$ for convenience. We have
\[\underset{X \in \M_n(\bZ/p^k\bZ)}\Prob(F(X - \bar{t}I_n) = 0) = \prod_{j=1}^{n} \underset{X \in \M_n(\bZ/p^k\bZ)}\Prob(F(Xv_j) = \bar{t}F(v_j)).\]

Since $Xv_j = \sum_{i=1}^{n}X_{ij}v_i$, we have $F(Xv_j) = \sum_{i=1}^{n}X_{ij}F(v_i)$, so

\begin{align*}
&\underset{X \in \M_n(\bZ/p^k\bZ)}\Prob(F(Xv_j) = \bar{t}F(v_j)) \\
&= \underset{X \in \M_n(\bZ/p^k\bZ)}\Prob\lt(\sum_{i=1}^{n}X_{ij}F(v_i) = \bar{t}F(v_j)\rt) \\
&= \underset{X \in \M_n(\bZ/p^k\bZ)}\Prob\lt(\sum_{i \in \sg}X_{ij}F(v_i) + \sum_{i \in [n] \sm \sg}X_{ij}F(v_i) = \bar{t}F(v_j)\rt) \\
&= \underset{X \in \M_n(\bZ/p^k\bZ)}\Prob\lt(\sum_{i \in [n] \sm \sg}X_{ij}F(v_i) = \bar{t}F(v_j) - \sum_{i \in \sg}X_{ij}F(v_i) \text{ and } \sum_{i \in \sg}X_{ij}F(v_i) \in \bar{t}F(v_j) + H\rt) \\
&= \underset{X \in \M_n(\bZ/p^k\bZ)}\Prob\lt(\sum_{i \in \sg}X_{ij}F(v_i) \in \bar{t}F(v_j) + H\rt) \\
&\hspace{3mm} \cdot \underset{X \in \M_n(\bZ/p^k\bZ)}\Prob\lt(\sum_{i \in [n] \sm \sg}X_{ij}F(v_i) = \bar{t}F(v_j) - \sum_{i \in \sg}X_{ij}F(v_i) \ \bigg| \ \sum_{i \in \sg}X_{ij}F(v_i) \in \bar{t}F(v_j) + H\rt).
\end{align*}

Now, we are in the same setting as in the proof of \cite[Lemma 2.7]{NV}, so we get the bound 

\[\underset{X \in \M_n(\bZ/p^k\bZ)}\Prob(F(Xv_j) = \bar{t}F(v_j)) \leq (1 - \ep)(D|G|^{-1} + e^{-\ep \dt n /p^{2k}}),\]

and repeating the proof of \cite[Lemma 2.8]{NV}, we get

\[\underset{X \in \M_{n}(\bZ/p^{k}\bZ)}{\Prob}(F(X - \bar{t}I_{n}) = 0) = \prod_{j=1}^{n} \underset{X \in \M_n(\bZ/p^k\bZ)}\Prob(F(Xv_j) = \bar{t}F(v_j)) \leq K_{G,\dt}e^{-\ep n} D^n |G|^{-n}.\] 
\end{proof}

\subsection{Proof of Theorem \ref{main4}} Finally, we prove Theorem \ref{main4} by proving (\ref{last}).

\begin{proof}[Proof of (\ref{last})] For any real number $0 < \dt \leq \ell(|G|)^{-1}/2$, we use Lemmas \ref{depths} and \ref{dpbd} so that
\begin{align*}
&\sum_{\substack{F \in \Sur_{R}(V, G)^{\#} : \\ F \text{ not code of distance }\dt n}} \lt|\underset{X \in \M_{n}(\bZ/p^{k}\bZ)}{\Prob}(F(X - \bar{t}I_{n}) = 0) - |G|^{-n}\rt| \\
&\leq \sum_{\substack{F \in \Sur_{R}(V, G)^{\#} : \\ F \text{ not code of distance }\dt n}} \lt(\underset{X \in \M_{n}(\bZ/p^{k}\bZ)}{\Prob}(F(X - \bar{t}I_{n}) = 0) + |G|^{-n}\rt) \\
&\leq \sum_{\substack{D \in \bZ_{>1}: \\ D | \#G}}\sum_{\substack{F \in \Sur_{R}(V, G)^{\#} : \\ F \text{ has $\dt$-depth } D}} \lt(\underset{X \in \M_{n}(\bZ/p^{k}\bZ)}{\Prob}(F(X - \bar{t}I_{n}) = 0) + |G|^{-n}\rt) \\
&\leq \lt(\sum_{\substack{D \in \bZ_{>1}: \\ D | \#G}}C_{G} K_{G,\dt}{n \choose \lceil \ell(D)\dt n \rceil - 1 }e^{-\ep n}|G|^{\ell(D)\dt n} \rt) + \lt( \sum_{\substack{D \in \bZ_{>1}: \\ D | \#G}}  C_{G}{n \choose \lceil \ell(D)\dt n \rceil - 1 }|D|^{-n + \ell(|G|) \dt n}\rt) \\
&\leq A_{G}C_{G} K_{G,\dt}{n \choose \lceil \ell(|G|)\dt n \rceil - 1} e^{-\ep n}|G|^{\ell(|G|)\dt n} + A_{G}C_{G}{n \choose \lceil \ell(|G|)\dt n \rceil - 1} 2^{-n + \ell(|G|) \dt n},
\end{align*}
where $A_{G} := \#\{D \in \bZ_{> 1} : D | \#G\}$ because $\lceil \ell(D)\dt n \rceil - 1 \leq \lceil \ell(|G|)\dt n \rceil - 1 < n/2$. We now bound the last two summands. The first summand is 

\begin{align*}
S_{1}(n) &:= A_{G}C_{G} K_{G,\dt}{n \choose \lceil \ell(|G|)\dt n \rceil - 1} e^{-\ep n}|G|^{\ell(|G|)\dt n} \\
&= A_{G}C_{G} K_{G,\dt}{n \choose \lceil \ell(|G|)\dt n \rceil - 1}  e^{-\ep n}e^{\log(|G|)\ell(|G|)\dt n} \\
&= A_{G}C_{G} K_{G, \dt}{n \choose \lceil \ell(|G|)\dt n \rceil - 1} e^{(-\ep + \log(|G|)\ell(|G|)\dt)n},
\end{align*}

and the second summand is 

\begin{align*}
S_{2}(n) &:= A_{G}C_{G}{n \choose \lceil \ell(|G|)\dt n \rceil - 1} 2^{-n + \ell(|G|) \dt n} \\
&= A_{G}C_{G}{n \choose \lceil \ell(|G|)\dt n \rceil - 1} e^{\log(2)(\ell(|G|) \dt - 1)n}
\end{align*}

Recall from the proof of Lemma \ref{cdbd2} that
\[{n \choose \lceil \ell(|G|)\dt n \rceil - 1} \leq {n \choose \lceil \ell(|G|)\dt n \rceil } \leq \exp\lt(nH\lt( \frac{\lceil \ell(|G|)\dt n \rceil}{n} \rt)\rt),\]
where $H(\al) := -\al \log(\al) - (1 - \al)\log(1 - \al)$ defined for real $\al \in (0, 1)$. Since $\lim_{\alpha \ra 0+}H(\al) = 0$ and $\lceil \ell(|G|) \dt n \rceil / n \leq 2\ell(|G|) \dt$ for all $n \geq (\ell(|G|)\dt)^{-1}$, we may take $\dt$ so small that
\[-\ep + \log(|G|)\ell(|G|)\dt + H\lt( \frac{\lceil \ell(|G|)\dt n \rceil}{n} \rt) \leq -\ep/2 + H\lt( \frac{\lceil \ell(|G|)\dt n \rceil}{n} \rt) < -\ep/4\]
and
\[\log(2)(\ell(|G|) \dt - 1) + H\lt( \frac{\lceil \ell(|G|)\dt n \rceil}{n} \rt) \leq -\log(2)/2 + H\lt( \frac{\lceil \ell(|G|)\dt n \rceil}{n} \rt) < -\log(2)/4,\]
for all $n \geq (\ell(|G|)\dt)^{-1}$. The first inequality shows that $\lim_{n \ra \infty}S_{1}(n) = 0$, and the second inequality shows that $\lim_{n \ra \infty}S_{2}(n) = 0$. This finishes the proof.
\end{proof}

\

\section*{Acknowledgments}

\hspace{3mm} We thank Roger Van Peski, Will Sawin, and Melanie Matchett Wood for helpful conversations. We thank Yifeng Huang, Nathan Kaplan, and Jungin Lee for helpful comments for an earlier version of this paper. Myungjun Yu was supported by the National Research Foundation of Korea (NRF) grant funded by the Korea government (MSIT) (No. 2020R1C1C1A01007604), by Korea Institute for Advanced Study (KIAS) grant funded by the Korea government, and by Yonsei University Research Fund (2022-22-0125).

\

\end{document}